\documentclass{amsart}
\usepackage{url}

\usepackage{graphicx}
\usepackage{datetime}
\usepackage{color}
\usepackage{amsmath,amsfonts,amsthm,amssymb}
\usepackage{epstopdf}
\usepackage[all,cmtip]{xy}
\usepackage{accents}
\newtheorem{theorem}{Theorem}
\newtheorem{prop}[theorem]{Proposition}

\newtheorem{lemma}[theorem]{Lemma}
\newtheorem{remark}[theorem]{Remark}
\newtheorem{conj}[theorem]{Conjecture}

\numberwithin{equation}{section} \numberwithin{theorem}{section}

\newcommand{\R}{\mathbb{R}}
\newcommand{\N}{\mathbb{N}}

\newcommand{\B}{\mathbb{B}}

\newcommand{\la}{\langle}
\newcommand{\ra}{\rangle}

\newcommand{\h}{\mathcal{H}}
\newcommand{\A}{\mathcal{A}}

\newcommand{\emb}{\hookrightarrow}

\newcommand{\In}{\subset}
\newcommand{\Om}{\Omega}
\newcommand{\om}{\omega}
\newcommand{\dl}{{\delta}}
\newcommand{\Dl}{{\Delta}}
\newcommand{\ta}{{\theta}}
\newcommand{\al}{{\alpha}}

\newcommand{\ed}{{\rm d}}

\newcommand{\D}{{\nabla}}
\newcommand{\Db}{{\nabla^{\bot}}}
\newcommand{\ti}[1]{{\tilde{#1}}}
\newcommand{\eps}{{\varepsilon}}
\newcommand{\fr}[2]{\frac{#1}{#2}}
\newcommand{\sm}{{\setminus}}

\newcommand{\T}{{\rm T}}
\newcommand{\p}{\rho}
\newcommand{\pl}[2]{{\frac{\partial #1}{\partial #2}}}

\newcommand{\zb}{\overline{z}}
\newcommand{\db}{\overline{\partial}}
\newcommand{\de}{\partial}
\newcommand{\nn}{\nonumber}
\newcommand{\mint}{\mathop{\int\hspace{-1.05em}{\--}}\nolimits}

\def\Xint#1{\mathchoice
   {\XXint\displaystyle\textstyle{#1}}%
   {\XXint\textstyle\scriptstyle{#1}}%
   {\XXint\scriptstyle\scriptscriptstyle{#1}}%
   {\XXint\scriptscriptstyle\scriptscriptstyle{#1}}%
   \!\int}
\def\XXint#1#2#3{{\setbox0=\hbox{$#1{#2#3}{\int}$}
     \vcenter{\hbox{$#2#3$}}\kern-.5\wd0}}

\def\dashint{\Xint-}

\begin{document}
\title[Global estimates and energy identities]{Global estimates and energy identities for elliptic systems with antisymmetric potentials}
\author{Tobias Lamm}
\address[T.~Lamm]{Institute for Analysis, Karlsruhe Institute of Technology (KIT), Englerstr. 2, 76131 Karlsruhe, Germany}
\email{tobias.lamm@kit.edu}
\author{Ben Sharp}
\address[B.~Sharp]{Imperial College London, Huxley Building, 180 Queen's Gate, SW7 2AZ, London, UK}
\email{ben.g.sharp@gmail.com}

%\date{\currenttime, \today}

\subjclass[2000]{}

\begin{abstract}
We derive global estimates in critical scale invariant norms for solutions of elliptic systems with antisymmetric potentials and almost holomorphic Hopf differential in two dimensions. Moreover we obtain new energy identities in such  norms for sequences of solutions of these systems. The results apply to harmonic maps into general target manifolds and surfaces with prescribed mean curvature. In particular our results confirm a conjecture of Rivi\`ere in the two-dimensional setting.
\end{abstract}

\maketitle

\section{Introduction}
Harmonic maps are one of the most studied solutions to a geometric partial differential equation. They are critical points of the Dirichlet energy
\[
E(u)=\frac12 \int_{M} |\nabla u|^2 dv_g,
\]
where $u\in W^{1,2}(M,N)$ is a map between between two Riemannian manifolds $(M^m,g)$ and $(N^l,h)$ and where we assume that $(N^l,h)$ is isometrically embedded in some euclidean space $\R^n$. The elliptic system satisfied by harmonic maps is 
\begin{align}
-\Delta u =A(u)(\nabla u, \nabla u), \label{harmonic}
\end{align}
where $A$ is the second fundamental form of the embedding $N\hookrightarrow \R^n$. When $m=2$ the Dirichlet energy, and thus harmonic maps, are invariant under conformal transformations of $M$ therefore this situation is of particular interest.  

It was discovered by H\'elein \cite{Helein} that in the local situation $M=B^m_1 \subset \R^m$ and $N=S^{n-1}$ the system \eqref{harmonic} can be written as 
\begin{align}
-\Delta u = \Db B \cdot \nabla u,
\end{align}
where $\Db B \in L^2(B_1^m,so(n)\otimes \bigwedge^1 \R^m)$ and where $\Db B \cdot \D u$ denotes an inner product of one forms coupled with matrix multiplication. In this situation $\text{div}\, \Db B =0$ weakly, and an extension of results due to Wente \cite{Wente} and M\"uller \cite{Mu} by Coifman et al. \cite{CLMS} allows to conclude that the right hand side $\Db B \cdot \D u$ is in the local Hardy space $\mathcal{H}^1_{loc}(B_1^m,\R^n)$, which is a strict subspace of $L^1(B_1^m,\R^n)$. In particular the improved regularity of the right hand side allows to conclude the global \emph{non-linear} estimate
\begin{align}
\| \nabla^2 u \|_{L^1_{loc}(B_1^m,S^{n-1})} \le C E(u).\label{globalsphere}
\end{align}
 %In particular if $\Om \cdot \D u\in L^1$ only, then the estimate \eqref{globalsphere} is not known to be true.  

When $m=2$ one infers from \eqref{globalsphere}, that $u\in C^0_{loc}(B_1^2,S^{n-1})$ and higher regularity then follows from standard results on elliptic systems. H\'elein showed that two-dimensional harmonic maps into general closed target manifolds are also smooth by rewriting the equation \eqref{harmonic} using a so called Coulomb frame, see \cite{helein91}, \cite{helein02}. We would like to remark that the construction of the Coulomb frame in this setting requires the target manifold to have a trivial tangent bundle, which (via a technical result) one may assume without loss of generality for sufficiently smooth targets. %Further to this, for arbitrary $N$ a linear estimate of the form
%\begin{align}
%\| \nabla^2 u \|_{L^1_{loc}(B_1^2,N)} \le C(N) \sqrt{E(u)}
%\end{align}  
%is known to hold but requires the Dirichlet energy to be small and hence one cannot directly conclude a global estimate which only depends on the energy of $u$. Using a naive approach by covering a compact subset of $B_1^2$ by balls in which the energy is small, yields a global estimate for any fixed solution, but with a constant depending on the inverse of the smallest radius used in the covering. By considering a sequence of harmonic maps where the energy locally concentrates, which is possible due to the conformal invariance of $E$, one readily observes that this smallest radius has to tend to zero and hence a global estimate obtained in this way degenerates along the sequence. 

For general closed target manifolds $N$ it was shown by Rivi\`ere \cite{Riv1} that the equation \eqref{harmonic} can be written in the form 
\begin{align}
-\Delta u = \Om \cdot \nabla u,\label{omega}
\end{align} 
with $\Om \in L^2(B_1^m,so(n)\otimes \bigwedge^1 \R^m)$ but it is no longer true that $\text{div}\, \Omega=0$, however $\Om$ remains anti-symmetric. For two-dimensional domains Rivi\`ere even showed that every critical point of a conformally invariant variational integral which is quadratic in the gradient satisfies an equation of the form \eqref{omega}.
His main result, an extension of H\'elein's regularity result, was that every weak solution of \eqref{omega} with small $L^2$-norm of $\Omega$ is as smooth as allowed by $\Omega$ - and in particular continuous. He obtained this result by deriving a conservation law which satisfies compensation properties via perturbing a Coulomb frame approach to studying \eqref{omega}. We remark that the anti-symmetry of $\Om$ is crucial to the improved regularity, moreover that this method applies to the harmonic maps setting without requiring any condition on the target manifold other than it being a $C^2$ sub-manifold of some Euclidean space. 

We further remark that \eqref{globalsphere} does not hold in general for solutions of \eqref{omega}, without the assumption that $\text{div}\, \Omega =0$ - for a counterexample see \cite{LR}.

Using a small energy assumption several interesting energy convexity, uniqueness and higher regularity results for harmonic maps and certain solutions of \eqref{omega} have recently been obtained in \cite{CM1}, \cite{lammlin}, \cite{sharp} and \cite{ST} using differing techniques. 

For $m\ge 3$ a partial regularity result for minimising harmonic maps was shown to hold by Schoen and Uhlenbeck \cite{SU}. For stationary harmonic maps this was proved by Evans \cite{Evans} for $N=S^{n-1}$ and by Bethuel \cite{bethuel} for general closed  target manifolds, using again the Coulomb frame. All of these results were generalised by Rivi\`ere and Struwe \cite{RivStruwe} by studying the system \eqref{omega} under appropriately motivated assumptions. 

Going back to harmonic maps for arbitrary closed targets $N$ a linear estimate of the form
\begin{align}
\| \nabla^2 u \|_{L^1_{loc}(B^m_1,N)} \le C(N) \sqrt{E(u)}
\end{align}  
is known to hold for weakly harmonic maps when $m=2$ (resp. weakly stationary harmonic maps for $m\geq 3$) but requires the Dirichlet energy to be small and hence one cannot directly conclude a global estimate which only depends on the energy of $u$. For arbitrary $u$ and $m=2$, one might try using a naive approach to cover a compact subset of $B_1^2$ by balls in which the energy is small, yielding a global estimate for any fixed solution, but with a constant depending on the inverse of the smallest radius used in the covering. By considering a sequence of harmonic maps where the energy concentrates, which is possible due to the conformal invariance of $E$, one readily observes that this smallest radius has to tend to zero and hence a global estimate obtained in this way degenerates along the sequence.

One of the main results of the present paper is to confirm the two-dimensional case of the following \textit{conjecture of Rivi\`ere} (see page $9$ in \cite{Riv1}): for every harmonic map  $u:B_2^m\to N \hookrightarrow \R^n$ with $E(u)\le \Lambda$, there exists a constant $C=C(m,N,\Lambda)$ so that
\begin{align}
\| \D^2 u\|_{L^1(B_1^m,N)} \le C. \label{estriv}
\end{align}
In the present paper we confirm this conjecture for $m=2$, and closed target manifolds $N$ of class $C^2$. We actually prove such a global estimate for general (approximate) solutions $u$ of \eqref{omega} whose Hopf differential $\phi$ is suitably controlled or almost holomorphic (see Theorems \ref{main theorem} and \ref{corollary theorem} for details). Recall that for a map $u\in W^{1,2}(B_1^2,N)$ the Hopf differential is defined by
$$\phi:= |u_x|^2 - |u_y|^2 - 2i\la u_x,u_y\ra\in L^1(B_1).$$
It follows from Rivi\`ere's regularity results that $\phi$ is weakly differentiable when $u$ solves \eqref{omega}, and a calculation shows that $\phi$ is holomorphic for every critical point of a conformally invariant variational integral which is quadratic in the gradient, and hence our result applies to all these critical points.
In particular, the global estimate holds for (approximate) harmonic maps into general target manifolds and conformally parametrised surfaces in $N$ of prescribed mean curvature $H\in L^\infty$ and even with $H\in L^2$. 
We remark that the special case of harmonic maps into a two-dimensional target manifold $N^2$ which is not diffeomorphic to $S^2$ has been solved previously by Rivi\`ere, see Theorem I.7 in \cite{Riv1}. 

For stationary harmonic maps for $m\geq 2$ into target manifolds which do not carry harmonic spheres $S^p$, $2\le p \le m$, Lin \cite{lin} showed a global estimate of the form
\[
\| \nabla u \|_{L^\infty(M,N)} \le C(M,N,E(u)).
\] 
This results applies in particular to target manifolds $N$ whose universal cover $\tilde N$ supports a pointwise convex function, as was also shown by Lin.

We note that for energy minimising harmonic maps into general target manifolds and in arbitrary dimensions, Cheeger and Naber \cite{cheegernaber} recently showed that there exists a number $p>2$ so that $u\in W^{1,p}\cap W^{2,p/2}(B_{1/2},N)$ with uniform bounds. This was improved in \cite{NV} where uniform bounds for $\nabla u$ in $L^{3,\infty}$ and $\nabla^2 u$ in $L^{3/2,\infty}$ were derived.

The importance of the estimate \eqref{estriv} stems from its applications in proving the so called energy identity during the bubbling process, which naturally occurs when one studies sequences of critical points of conformally invariant functionals or their corresponding gradient flows. A ``bubble" is formed when a certain threshold of energy concentrates on shrinking discs along a sequence - see \cite{SacksUhlenbeck} for a first description of this phenomenon in the context of harmonic maps. In fact the bubbling process can be thought of as a covering-type argument where one attempts to keep track of potential energy concentration along sequences of solutions - the difficulty then being to gain a suitable estimate on each component of this covering - which consists of shrinking discs, their complement (which is a multiply connected domain) and an intermediate region formed of degenerating annuli. Good control is obtained locally on the shrinking discs and multiply connected domains via the $\eps$-regularity results of H\'elein or Rivi\`ere. Thus the only place left to control is the degenerating annuli, or connecting neck regions. 
In order to explain this further we need to introduce two more function spaces, the Lorentz spaces $L^{2,1}$ and $L^{2,\infty}$, the latter one is also called the weak-$L^2$ or Marcinkiewicz-space. For $U\In \R^2$ and measurable $f$ let $\lambda_f(s):=  |\{x\in U : |f(x)|>s\}|$ which is non-increasing in $s> 0$. We define:% we have
%$$\|f\|_{L^2(U)}^2 = \int_0^{\infty} 2s\lambda_f (s)\ed s .$$ It should be easy to see that this implies $\lambda_f(s) \leq \|f\|_{L^2(U)}^2 s^{-2}$ which motivates the definition of 
%$$L^{2,\infty}(U) = \{ f\in L^1_{loc} : \lambda_f (s)\leq Cs^{-2}\}$$ 
$$
L^{2,1}(U)= \{ f : \int_0^{\infty} \lambda_f(s)^{\fr12} \ed s < \infty \}\,\,\,\,\text{and}\,\,\,\,
L^{2,\infty}(U)= \{ f:\sup_{s> 0} s^2\lambda_f(s)<\infty\}
$$
and we note that $L^{2,1}(U) \In L^2(U) \In L^{2,\infty}(U)$. This can be easily checked once one considers 
$$\|f\|_{L^2(U)}^2 = \int_0^{\infty} 2s\lambda_f (s)\ed s.$$
There are norms associated with the above spaces which are equivalent to the quantities appearing in the definition. Moreover it should be at least intuitively clear from the definition that 
\begin{align}
\left|\int_U fg \right|\leq \|f\|_{L^{2,1}(U)}\|g\|_{L^{2,\infty}(U)}\label{duality},
\end{align} 
which can be summed up by $(L^{2,1})^{\ast} = L^{2,\infty}$, see \cite{hunt66}.  

Going back to the bubbling picture, when studying a sequence of say harmonic maps $u_k$ from two-dimensional domains with uniformly bounded energy % it is nowadays well-known that %the phenomenon of bubbling can occur, i.e. 
 a certain amount of energy can concentrate at finitely many points and disappear when taking weak limits $u_k\rightharpoonup u$. By performing suitable re-scalings (called blow-up's) one can recapture this lost energy and re-discover it as the energy of a non-trivial harmonic map from $\R^2 \to N$ (a so called bubble). 
Away from the finitely many points where the energy concentrates one concludes from standard small-energy regularity results that the maps $u_k$ are very close to the weak limit $u$ and in small degenerating balls around the energy concentration points, $u_k$ is very close to the bubbles. There is some intermediate region between the two sets on which we have good control on the $u_k$'s and it can be shown to consist of degenerating annuli, i.e. annuli for which the quotient of the outer radius divided by the inner radius diverges. 
In order to show that there is no unaccounted energy loss (i.e the energy identity) in this process, one has to show that the energy converges to zero on these degenerating annuli. Now it follows from standard small energy regularity results that the $L^{2,\infty}$-norm of $\nabla u_k$ has to tend to zero on the annuli. Hence, using \eqref{duality}, it remains to derive a uniform bound on the $L^{2,1}$-norm of $\nabla u_k$ on the annuli and this fact indeed follows from \eqref{estriv} using an extension of the classical Sobolev embedding theorem (see e.g. \cite{Helein}). To our knowledge Lin and Rivi\`ere \cite{lin01}, \cite{lin01a}, \cite{lin02b} were the first ones to observe the importance of the duality \eqref{duality} in this setting. Indeed, in \cite{lin02b} they used this idea to derive a type of energy identity for sequences of stationary harmonic maps from higher dimensional domains into spheres and an extension of \eqref{estriv} to this setting would have direct applications to obtaining a corresponding result for general targets.
We note that the same idea has later been used, see \cite{LiZhu}, \cite{lamm06}, \cite{lin02}, \cite{Zhu} in the setting of harmonic maps, \cite{riviere02} for Yang-Mills fields and \cite{bernardriviere} for Willmore surfaces. 

Since this bubbling process is a crucial ingredient in our argument, we include a detailed presentation of it in the appendix. Similar results can be found in \cite{DingTian}, \cite{Parker}, \cite{LR} and \cite{bernardriviere}.

Our second main result consists of energy identities for sequences of approximate solutions of \eqref{omega} with almost holomorphic Hopf differentials (see Theorem \ref{EIL21}). These results extend previous work of \cite{DingTian}, \cite{duzaar98}, \cite{Jost2DBook}, \cite{LiZhu}, \cite{LinWang}, \cite{lin02}, \cite{Parker}, \cite{qing95}, \cite{WangPS} and \cite{zheng}, in which various versions of the energy identity have been proved for sequences of approximate harmonic maps or other systems of the type \eqref{omega}.
We also want to mention that Laurain and Rivi\`ere \cite{LR} recently showed an energy identity for the angular derivative of sequences of solutions of \eqref{omega} without assuming a condition on the Hopf differential. Moreover, they constructed a counterexample which shows that the full energy identity cannot be true without additional assumptions, such as the almost holomorphicity of the Hopf differential.

We also highlight that our main supporting Theorem \ref{secder} yields an energy identity in terms of the $L^{2,1}$ norms of the gradients and as a direct consequence we conclude that the no-neck property holds, i.e. the weak limit $u$ and all the bubbles are connected without necks. This fact generalises the results of \cite{ChenTian}, \cite{QingTian} and \cite{Zhu} to our more general setting.
Imposing a structural condition on $\Omega$, which is for example satisfied by solutions of \eqref{harmonic}, we also derive an energy identity in terms of the $L^1$-norms of the second derivatives.

Finally we study solutions to \eqref{omega} on Riemann surfaces under the condition that the Hopf differential is holomorphic. We draw attention to the fact that such solutions are conformally invariant which allows for a general study of these solutions on a sequence of potentially degenerating Riemann surfaces, similarly to the harmonic map setting \cite{Zhu_M}. We do not go into full details here, however the critical analysis on a degenerating Riemann surface is along conformally long cylinders - or degenerating annuli, for which our main supporting Theorem \ref{secder} can be applied.  We also link this to the study of $W^{2,2}$ conformal immersions (see \cite{Kuwert_Li}, \cite{LR2}).  

An outline of the paper is as follows: In section 2 we state our main results and the most important supporting result. In section 3 we derive estimates relating the radial derivative of a map and its Hopf differential together with the angular derivative. We also extend various results on harmonic functions and Wente-type equations on annuli of \cite{LR} to our setting. The proof of the global estimate is contained in section 4 and in section 5 we prove the energy identities. In section 6 we make some remarks on the equation \eqref{omega} on Riemann surfaces. The bubbling argument which is crucial to us can be found in the appendix.\\

{\bf{Acknowledgements}} The second author was funded by Andr\'e Neves' European Research Council STG agreement number P34897 during the writing of this paper. He would also like to thank Karlsruhe Institute of Technology for their kind hospitality during the preliminary stages of this project.

\section{Results}
Our first main result is the global $W^{2,1}$-estimate for solutions of \eqref{omega} under further control on the Hopf differential. 
\begin{theorem}\label{main theorem}
Let $B_1\In \R^2$ be the unit ball and consider $u\in W^{1,2}(B_1,\R^n)$, $f \in L\log L(B_1,\R^n)$ and $\Om \in L^2(B_1,so(n)\otimes \bigwedge^1 \R^2)$ solving 
\begin{eqnarray*}
-\Dl u = \Om \cdot \D u + f. %\label{first} \\
%0&=& \Om \cdot \Db u + g.  \label{second}
\end{eqnarray*}
We will also assume that the Hopf differential 
$$\phi:= |u_x|^2 - |u_y|^2 - 2i\la u_x,u_y\ra\in L^1(B_1)$$ satisfies 
$|\phi|^{\fr12}\in L^{2,1}_{loc}(B_1)$. 
Then for every compact subset $K\subset B_1$, there exists some $$C=C(K,\|\Om\|_{L^2(B_1)}, \|\D u\|_{L^2(B_1)}, \|f\|_{L\log L(B_1)}, \||\phi|^{\fr12}\|_{L^{2,1}(K)})<\infty$$ such that 
$$\|\D^2 u\|_{L^{1}(K)} +\| \D u \|_{L^{2,1}(K)}\leq C.$$ 
\end{theorem}
\begin{remark}
\begin{itemize}
\item[1)] Recall that the space $L\log L(B_1,\R^n)$ is defined by
\[
L \log L(B_1,\R^n):= \{ f:B_1\to\R^n | \, \int_{B_1} |f(x)| \log (2 +|f(x)|)dx < \infty\}.
\]
Hence the condition that $f\in L\log L$ can be thought of as a borderline between $f\in L^1$ and $f\in L^p$ for $p>1$.
\item[2)] We note again that this result cannot be deduced from standard small energy regularity results together with a covering argument, since in this case the constant would also depend on the inverse of the infimum of all radii such that the small energy regularity result is applicable. But this infimum can be arbitrary small by considering a sequence of solutions of the above system which allows bubbling.
\item[3)]
Notice that, \textit{a-priori}, $|\phi|^{\fr12} \in L^2(B_1)$ so we really do require more regularity for the Hopf differential than is given by the assumptions on $u$. However this improved regularity for $\phi$ is easily obtained in the vast majority of situations - for instance if the Hopf differential is almost holomorphic, see  Proposition \ref{omegaperp}. 
\item[4)] As already stated, this theorem is not true without the extra control on $\phi$ - see \cite{LR}. 
\item[5)] An open question here is whether one can replace $L\log L$ by the local Hardy space $\h^1_{loc}$ and still get $W^{2,1}$ control - even in the case that $\|\Om\|_{L^2}$ is small. 
\end{itemize}
\end{remark}
A corollary of the above theorem is the following: 
\begin{theorem}\label{corollary theorem}
Let $B_1\In \R^2$ be the unit ball and consider $u\in W^{1,2}(B_1,\R^n)$, $f,g \in L^2(B_1,\R^n)$ and $\Om \in L^2(B_1,so(n)\otimes \bigwedge^1 \R^2)$ solving 
\begin{eqnarray}
-\Dl u &=& \Om \cdot \D u + f \label{first} \\
0&=& \Om \cdot \Db u + g.  \label{second}
\end{eqnarray}
Then for every compact subset $K\subset B_1$, there exists some 
\[
C=C(K,\|\Om\|_{L^2(B_1)}, \|\D u\|_{L^2(B_1)}, \|f\|_{L^2(B_1)}, \|g\|_{L^2(B_1)})<\infty
\] 
such that 
$$\|\D^2 u\|_{L^{1}(K)} +\| \D u \|_{L^{2,1}(K)}\leq C.$$ 
\end{theorem}
In order to prove this theorem, we show in section 4, that under these assumptions $\db \phi \in L^1(B_1)$, from which we derive the necessary regularity in order to be able to apply Theorem \ref{main theorem}. We remark that in all known geometric applications of this result the second equation \eqref{second} holds for $g\equiv 0$. 
\begin{remark}\label{dbaretc}
Interpreting $\Om$ as being connection forms for the trivial pull-back bundle $u^{\ast} (\T \R^n)$ we could re-write the PDE system \eqref{first} and \eqref{second} as 
\begin{eqnarray}
\dl_{\Om} (\ed u) &:=& \delta d u - \star (\Omega \wedge \star du)= f \label{firstc}\\
\ed_{\Om} (\ed u) &:=& d (du) + \Omega \wedge du= \ast g \label{secondc}. 
\end{eqnarray}
Or more succinctly with the standard complex structure on $\R^2$: 
$$-\db_{\Om} (\de u):=-\db (\de u) - \Om^{\zb}\wedge \de u = \fr14(f+ig).$$

The second condition $\ast(\Om\wedge \ed u) = -\Om\cdot \Db u = g$ is satisfied in all known geometric applications of this theorem for $g\equiv 0$. 
\end{remark}

The next result is our main supporting theorem, which we couple with the estimates for harmonic functions on cylinders (c.f. Proposition \ref{harm}) and the bubbling argument in the appendix in order to prove Theorem \ref{main theorem}.
\begin{theorem}\label{secder}
There exists an $\eps>0$ such that for all $\lambda,r>0$ satisfying $2r<1$, $\lambda<1$ and $\Om\in L^2(B_1\sm B_r, so(n)\otimes \bigwedge^1 \R^2)$, $f\in L\log L (B_1\sm B_r)$, $u\in W^{1,2}(B_1\sm B_r,\R^n)$ with $|\phi|^{1/2} \in L^{2,1}(B_1\sm B_r)$ satisfying 
\begin{align*}
-\Dl u =& \Om \cdot \D u + f\,\,\,\,\,\,\,\,\,\,\,\text{on $B_1\sm B_r$}
\end{align*}
and 
\begin{align*}
\sup_{r<\rho<\fr{1}{2}} \int_{B_{2\rho}\sm B_{\rho}} |\Om|^2 \leq& \eps,
\end{align*}
there exists some $C=C(\lambda,n)<\infty$ such that 
\begin{align*}
\left\| \nabla^2 u \right\|_{L^{1}(B_{\lambda }\sm B_{r/\lambda})} +\| \D u \|_{L^{2,1}(B_{\lambda }\sm B_{r/\lambda})}\leq& C\Big((1+\|\Om\|_{L^2(B_1\sm B_r)})(\|\D u\|_{L^2(B_1\sm B_r )}\\
&+\|f\|_{L\log L(B_1\sm B_r)})\Big)\\
&+C\||\phi|^{1/2}\|_{L^{2,1}(B_1\sm B_r)}.
\end{align*}
\end{theorem}
A similar result has been obtained by Laurain and Rivi\`ere \cite{LR} without an assumption on $\phi$. But on the other hand, they only conclude an estimate for the $L^{2,1}$-norm of the angular part of the first derivative of $u$. Moreover a counter-example from \cite{LR} serves to show that our estimate is false without the condition on $\phi$. 

The argument required to prove Theorem \ref{main theorem} will be by contradiction coupled with Theorem \ref{secder}; we perform a bubbling argument and show that the only way a uniform bound as in Theorem \ref{main theorem} can fail to hold is if the norm blows up on so called ``neck domains" which are precisely of the form of Theorem \ref{secder}. 

Finally, we also mention our main new energy identity and no-neck property result.

\begin{theorem}\label{EIL21}
Let $u_k\in W^{1,2}(B_1,\R^n)$ be a sequence of solutions of 
\begin{align*}
-\Dl u_k =& \Om_k \cdot \D u_k + f_k, 
\end{align*}
where $\Om_k \in L^2(B_1\sm B_r, so(n)\otimes \bigwedge^1 \R^2)$, $f_k\in L \log L(B_1,\R^n)$, and we assume that there exists a constant $\Lambda>0$ so that for every $k\in \N$
\[
\int_{B_1} \left( |\D u_k|^2+|\Om_k|^2 \right)dx+\| f_k\|_{L\log L(B_1)}+ \| |\phi_k|^{1/2}\|_{L^{2,1}(B_1)} \le \Lambda.
\]
Then there exists a subsequence, still denoted by $u_k$, $\Om_k$ and $f_k$, so that $u_k \rightharpoonup u $ weakly in $W^{1,2}(B_1)$, $\Omega_k \rightharpoonup \Om$  weakly in $L^2(B_1)$ and $f_k\rightharpoonup f \in L\log L(B_1)$ in a distributional sense and the limits are solutions of 
\begin{align*}
-\Dl u =& \Om \cdot \D u + f.
\end{align*}
Moreover there exist at most finitely many $\omega$-bubbles $\omega^{i,j}:\R^2\to \R^n$, $1\leq i \leq p$, $1\leq j\leq j_i$, i.e. solutions of
\[
-\Delta \omega^{i,j} = \Om^{i,j} \cdot \D \omega^{i,j},%,\ \ \ \Om^{i,j} \cdot \nabla^\perp \omega^{i,j}=0,
\]
sequences of points $x_k^{i,j} \in B_1$, $x_k^{i,j} \to x_i$, and sequences of radii $t_k^{i,j}\in \R_+$, $t_k^{i,j} \to 0$, such that for every $r<1$ so that $\{ x_1,\ldots,x_p\}\in  B_r$ 
\begin{align}
 \max \{  \frac{t_k^{i,j}}{t_k^{i,j'}},\frac{t_k^{i,j'}}{t_k^{i,j}},\frac{\text{dist}(x_k^{i,j},x_k^{i,j'})}{t_k^{i,j}+t_k^{i,j'}} \} &\to \infty, \ \ \ \forall \ \ 1\leq i \leq p, \ \ 1\leq j,j'\leq j_i,\ \ j\not= j', \label{conv1}\\  
\lim_{k\to \infty} \|\nabla u_{k}\|_{L^{2}(B_r,\R^n)}^2&=\| \nabla u\|^2_{L^{2}(B_r,\R^n)}+\sum_{i=1}^p\sum_{j=1}^{j_i} \|\nabla \omega^{i,j}\|_{L^{2}(\R^2,\R^n)}^2. \label{energyequality9}
\end{align}
If we assume additionally that $f_k\in L^p(B_1,\R^n)$ for some $1<p\le \infty$ with 
\[
\| f_k\|_{L^p(B_1)} \le \Lambda
\]
and
\[
\| |\phi_k|^{1/2}\|_{L^{2,1}(Z)}  \to 0
\]
for every subset $Z\subset B_1$ with $|Z| \to 0$, then we also have for every $r<1$ as above
\begin{align}
\lim_{k\to \infty} \|\nabla u_{k}\|_{L^{2,1}(B_r,\R^n)}^2&=\| \nabla u\|^2_{L^{2,1}(B_r,\R^n)}+\sum_{i=1}^p\sum_{j=1}^{j_i} \|\nabla \omega^{i,j}\|_{L^{2,1}(\R^2,\R^n)}^2.  \label{energyequality10}
\end{align}
Furthermore, the map $u$ and the maps $\omega^{i,j}$ are connected without necks and $\om^{i,j}$ are all conformal.  
\end{theorem}
\begin{remark}
Once again an easy corollary of this theorem is in the setting where $f_k \in L^2(B_1,\R^n)$ and additionally $\Om_k \cdot \Db u_k = g_k \in L^2(B_1,\R^n)$ (with uniformly bounded norms) under which all the additional assumptions on the Hopf differential are true. In particular, the result applies to sequences of critical points of conformally invariant variational problems with quadratic growth in the gradient.
\end{remark}

It is known that for general solutions to \eqref{omega} we cannot expect better than $W^{2,p}_{loc}$ regularity for $p<2$ - see \cite{Sh}, however we make the following
\begin{conj}
If $u$ is a solution to \eqref{omega} such that $\phi = 0$ almost everywhere, then $u\in W^{2,2}\cap W^{1,\infty}$ - in particular $u$ could be said to have weak mean curvature in $L^2$. 

We also remark that an interesting question here is whether or not the zeros of $\D u$ are finite and isolated under these conditions. 
\end{conj}

\section{Supporting results}
In this section we collect all results which are needed in order to prove the main Theorems mentioned before.

\subsection{Estimates involving the Hopf differential}
In the following we do some computations in polar coordinates $(\rho,\theta)$: 
Obviously for a map $u\in W^{1,2}(B_1)$ we have
\[
u_x = u_{\rho}\fr{x}{\p} - u_{\ta}\fr{y}{\p^2}, \ \ \  u_y = u_{\p}\fr{y}{\p} + u_{\ta}\fr{x}{\p^2} \ \ \ \text{and} \ \ \ |\D u|^2 = |u_x|^2 + |u_y|^2 
= |u_{\p}|^2 +\fr{|u_{\ta}|^2}{\p^2} .
\]
%and therefore we conclude
%\[
%|\D u|^2 = |u_x|^2 + |u_y|^2 
%= |u_{\p}|^2 +\fr{|u_{\ta}|^2}{\p^2} .
%\]
Recall that for $a,b \geq 0$ we have
\[
(a^{\fr12} + b^{\fr12})(a+b)^{\fr12} = (a^2 + ab)^{\fr12} + (b^2 + ab)^{\fr12} \geq a +b 
\]
and hence $(a+b)^{\fr12} \leq a^{\fr12} + b^{\fr12}$ so %from \eqref{1} 
we have 
\begin{equation}\label{2}
|u_\p| = \left(  |u_{\p}|^2 -\fr{|u_{\ta}|^2}{\p^2}  + \fr{|u_{\ta}|^2}{\p^2}\right)^{\fr12} \leq \left( \left| |u_{\p}|^2 -\fr{|u_{\ta}|^2}{\p^2} \right|\right)^{\fr12} + \fr{|u_{\ta}|}{\p}.
\end{equation}
\begin{prop}\label{hopfest}
Let $u\in W^{1,2}(B_1)$ and let $\phi$ be the Hopf differential of $u$, then% $\cap L^\infty(B_1,\R^n)$
\begin{equation*}
|u_\p| \leq |\phi|^{\fr12} + \fr{|u_{\ta}|}{\p}.
\end{equation*}
\end{prop}
\begin{proof}
By formula \eqref{2} we see that it suffices to show that
$$\left||u_{\p}|^2 -\fr{|u_{\ta}|^2}{\p^2}\right| \leq |\phi|.$$
From the above formulas for the partial derivatives of $u$ we have 
$$2\la u_x, u_y\ra = 2\fr{xy}{\p^2}\left( |u_{\p}|^2 -\fr{|u_{\ta}|^2}{\p^2} \right) + 2\fr{x^2-y^2}{\p^2}\left\la u_{\p}, \fr{u_{\ta}}{\p}\right\ra$$ and thus
\begin{eqnarray}\label{3}
(2\la u_x, u_y\ra)^2 &=& 4\fr{x^2y^2}{\rho^4}\left( |u_{\p}|^2 -\fr{|u_{\ta}|^2}{\p^2} \right)^2 + 4\fr{(x^2-y^2)^2}{\p^4}\left\la u_{\p}, \fr{u_{\ta}}{\p}\right\ra^2 +\nn\\&&+ 8\fr{xy(x^2-y^2)}{\p^4}\left( |u_{\p}|^2 -\fr{|u_{\ta}|^2}{\p^2} \right)\left\la u_{\p}, \fr{u_{\ta}}{\p}\right\ra.
\end{eqnarray}

Moreover 
$$|u_x|^2 - |u_y|^2 = \fr{x^2-y^2}{\p^2}\left( |u_{\p}|^2 -\fr{|u_{\ta}|^2}{\p^2} \right)-4\fr{xy}{\p^2}\left\la u_{\p}, \fr{u_{\ta}}{\p}\right\ra$$
giving 
\begin{eqnarray}\label{4}
(|u_x|^2 - |u_y|^2)^2 &=& \fr{(x^2 -y^2)^2}{\p^4}\left( |u_{\p}|^2 -\fr{|u_{\ta}|^2}{\p^2} \right)^2 + 16\fr{x^2y^2}{\p^4}\left\la u_{\p}, \fr{u_{\ta}}{\p}\right\ra^2 +\nn\\
&& - 8\fr{xy(x^2-y^2)}{\p^4}\left( |u_{\p}|^2 -\fr{|u_{\ta}|^2}{\p^2} \right)\left\la u_{\p}, \fr{u_{\ta}}{\p}\right\ra.
\end{eqnarray}
Putting together \eqref{3} and \eqref{4} gives
\begin{eqnarray}
|\phi|^2 &=& (|u_x|^2 - |u_y|^2)^2 + (2\la u_x, u_y\ra)^2 \nn\\
&=& \fr{(x^2-y^2)^2 + 4x^2y^2}{\p^4}\left( |u_{\p}|^2 -\fr{|u_{\ta}|^2}{\p^2} \right)^2 + \fr{4(x^2 - y^2)^2 + 16x^2y^2}{\p^4}\left\la u_{\p}, \fr{u_{\ta}}{\p}\right\ra^2\nn \\
&=&  \left( |u_{\p}|^2 -\fr{|u_{\ta}|^2}{\p^2} \right)^2 + 4\left\la u_{\p}, \fr{u_{\ta}}{\p}\right\ra^2 \nn \\
&\geq & \left( |u_{\p}|^2 -\fr{|u_{\ta}|^2}{\p^2} \right)^2
\end{eqnarray}
and this finishes the proof.
\end{proof}

\subsection{Harmonic functions on conformally long cylinders}
Next we derive some estimates for harmonic functions on conformally long cylinders. The importance of these estimates for deriving results similar to Theorem \ref{secder} was made clear by Laurain and Rivi\`ere \cite{LR}.
\begin{prop}\label{harm}
Let $h$ be a harmonic function on $B_1\sm B_{\eps}$ for some $\eps< \fr14$. Then for any $0< \lambda <\fr12 $, there exists $C=C(\lambda)$ such that  
\begin{align}
\left\| \frac{d}{dr}(rh_r)\right\|_{L^{2,1}(B_{\lambda}\sm B_{\eps \lambda^{-1}})} + \left\|\fr{h_{\theta}}{r}\right\|_{L^{2,1}(B_{\lambda}\sm B_{\eps \lambda^{-1}})} +& \left\|\fr{h_{\theta\theta}}{r}\right\|_{L^{2,1}(B_{\lambda}\sm B_{\eps \lambda^{-1}})} \nonumber \\
&+ \|h_{r\theta}\|_{L^{2,1}(B_{\lambda}\sm B_{\eps \lambda^{-1}})} \nonumber \\
\leq& C\| \nabla h\|_{L^2(B_1\sm B_{\eps})}.\label{estharmonic}
\end{align}
Moreover, we have the estimate
\begin{align}
\|\D^2 h\|_{L^1(B_{\lambda}\sm B_{\eps \lambda^{-1}})} \leq C (\| h_r\|_{L^{2,1}(B_{\lambda}\sm B_{\eps\lambda^{-1}})}+\| \nabla h\|_{L^{2}(B_1\sm B_{\eps})}).\label{estharmonic2}
\end{align}
\end{prop}
\begin{proof}
We write the harmonic function $h$ as 
$$h(r,\theta) = c_0 + d_0 \log r + \sum_{n\in \mathbb{Z}\sm\{0\}} (c_nr^n + d_n r^{-n})e^{in\theta}.$$
Thus we have 
$$\left|\fr{h_{\theta}(r,\theta)}{r} \right| \leq \sum_{n\in \mathbb{Z}\sm\{0\}}\left( |nc_n|r^{n-1} + |nd_n|r^{-n-1}\right)$$
and also 
$$\left|\fr{h_{\theta\theta}(r,\theta)}{r}\right| + |h_{r\theta}(r,\theta)| \leq 2\sum_{n\in \mathbb{Z}\sm\{0\}}\left( |n^2 c_n| r^{n-1} + |n^2 d_n| r^{-n-1}\right).$$
Setting 
$$H:= \sum_{n\in \mathbb{Z}\sm\{0\}} \left(|n^2 c_n| r^{n-1} + |n^2 d_n| r^{-n-1}\right)$$
we observe that the estimate 
\[
\|H\|_{L^{2,1}(B_{\lambda}\sm B_{\eps \lambda^{-1}})} \leq C\|\nabla h\|_{L^2(B_1\sm B_{\eps})}
\] 
would show \eqref{estharmonic} for the last three terms on the left hand side. 

We can estimate each term of $H$ in $L^{2,1}$ and we have, as in the appendix of \cite{LR}, 
$$\|r^{n-1}\|_{L^{2,1}(B_{\lambda}\sm B_{\eps \lambda^{-1}})}\leq\sqrt{\pi}\lambda^n$$ and 
$$\|r^{-n-1}\|_{L^{2,1}(B_{\lambda}\sm B_{\eps \lambda^{-1}})}\leq 2\sqrt{\pi}\left(\fr{\lambda}{\eps}\right)^{n} $$
when $n\geq 1$. 

Therefore 
\begin{align*}
\|H\|_{L^{2,1}(B_{\lambda}\sm B_{\eps \lambda^{-1}})}\leq& 2\sqrt{\pi} \Big( \sum_{n\geq 1} n^2\lambda^n (|c_n| + |d_{n}|\eps^{-n}) \\
&+ \sum_{n\leq -1} n^2\lambda^{-n}(|c_n|\eps^{n} + |d_{n}|) \Big)\nonumber\\
\leq&\left(\sum_{n\in \mathbb{Z}\sm\{0\}} |n|^3(2\lambda)^{2|n|}\right)^{\fr12}\times \nonumber\\
&\times\left(\sum_{n\leq -1}|n|\eps^{-2|n|}(c_n^2 + d_{-n}^2) + \sum_{n>0} |n|2^{-n}(c_n^2 + d_{-n}^2)\right)^{\fr12}\nonumber \\
\leq& C\|\nabla h\|_{L^2(B_1\sm B_{\eps})}.
\end{align*}
In order to estimate the first term on the left hand side of \eqref{estharmonic}, we note that it follows from the previous estimate, since
\[
\frac1{r} \frac{d}{dr} (rh_r) =-\frac{h_{\theta \theta}}{r^2}
\]
as $h$ is harmonic.

In particular, we can use \eqref{estharmonic} and the duality of the Lorentz spaces $L^{2,1}$ and $L^{2,\infty}$, in order to get
\begin{align*}
\| h_{rr}\|_{L^1(B_{\lambda}\sm B_{\eps \lambda^{-1}})} \le& c \|r^{-1}\|_{L^{2,\infty}(B_{\lambda}\sm B_{\eps \lambda^{-1}})} \| rh_{rr} \|_{L^{2,1}(B_{\lambda}\sm B_{\eps \lambda^{-1}})}\\
\le& c (\| h_r\|_{L^{2,1}(B_{\lambda}\sm B_{\eps\lambda^{-1}})}+\|\nabla h\|_{L^2(B_1 \sm B_\eps)}) .
\end{align*}

In order to show \eqref{estharmonic2}, we note that 
$$h_{xx} = h_{rr}\fr{x^2}{r^2} + h_r\fr{y^2}{r^3} + h_{\theta\theta}\fr{y^2}{r^4} + h_{\theta}\fr{2xy}{r^4} - 2h_{r\theta}\fr{xy}{r^3},$$
$$h_{yy} = h_{rr}\fr{y^2}{r^2} + h_r\fr{x^2}{r^3} + h_{\theta\theta}\fr{x^2}{r^4} - h_{\theta}\fr{2xy}{r^4} + 2h_{r\theta}\fr{xy}{r^3},$$
and 
$$h_{xy} = h_{rr}\fr{xy}{r^2} - h_r\fr{xy}{r^3} - h_{\theta\theta}\fr{xy}{r^4} + h_{\theta}\fr{y^2-x^2}{r^4} + h_{r\theta}\fr{x^2-y^2}{r^3}.$$
Hence we get
\[
|\nabla^2 h| \le C(|h_{rr} |+\fr{|h_r|}{r} +\fr{|h_{\theta\theta}|}{r^2} + \fr{|h_{\theta}|}{r^2} + \fr{|h_{r\theta}|}{r})
\]
and using the same duality argument as above, combined with \eqref{estharmonic}, we get \eqref{estharmonic2}.
\end{proof}

\subsection{Wente estimates on annuli}
In this subsection we use the above estimates for harmonic functions to derive new Wente estimates on annuli.
\begin{lemma}\label{wente}
Let $r<1/4$, $a,b\in W^{1,2}(B_1)$ and let $\psi\in W^{1,2}_0(B_1 \backslash B_r)$ be a solution of 
\[
\Delta \psi = \nabla a \cdot \nabla^\perp b
\]
on $B_1 \backslash B_r$, where $\nabla^\perp :=(-\partial_y,\partial_x)$. Then, for every $r <\lambda <\frac12 $ we have that $\nabla^2 \psi \in W^{2,1}(B_\lambda \backslash B_{r \lambda^{-1}})$ and there exists a constant $C(\lambda)$ so that
\begin{align}
\| \nabla^2 \psi \|_{L^1(B_\lambda \backslash B_{r \lambda^{-1}})} + \| \nabla \psi \|_{L^{2,1}(B_\lambda \backslash B_{r \lambda^{-1}})}\le C(\lambda)\| \nabla a\|_{L^2(B_1)} \| \nabla b\|_{L^2(B_1)}. \label{wenteannulus} 
\end{align}
\end{lemma}
\begin{proof}
The estimate for the $L^{2,1}$-norm of $\nabla \psi$ can be found in Lemma 2.1 of \cite{LR}. Hence it remains to show the $L^1$-estimate for $\nabla^2 \psi$ and for this we first consider the unique solution $\varphi\in W^{1,2}_0(B_1)$ of
\[
\Delta \varphi =\nabla a \cdot \nabla^\perp b.
\]
It follows from the results in \cite{CLMS} that
\[
 \| \nabla \varphi \|_{L^2(B_1)} +\| \nabla^2 \varphi \|_{L^1(B_1)} \le C \| \nabla a\|_{L^2(B_1)} \| \nabla b\|_{L^2(B_1)}.
\]
Next we let $w$ be the harmonic function with $w|_{\partial B_1} =0$ and $w|_{\partial B_r}=-\varphi$. 
It was shown by Laurain and Rivi\`ere (see the proof of Lemma 2.1 in \cite{LR}) that for every $r <\lambda <1$
\[
\| \nabla w\|_{L^2(B_1\backslash B_{r})}+\|\nabla w \|_{L^{2,1}(B_1\backslash B_{r \lambda^{-1}})}\le C(\lambda)\| \nabla a\|_{L^2(B_1)} \| \nabla b\|_{L^2(B_1)}.
\]
Combining this with Proposition \ref{harm} we get that for every $\lambda$ as in the statement of the Lemma
\[
\| \nabla^2 w \|_{L^1(B_\lambda \backslash B_{r \lambda^{-1}})}\le C(\lambda)\| \nabla a\|_{L^2(B_1)} \| \nabla b\|_{L^2(B_1)}.
\]
Since $\psi=\varphi+w$ the above estimates imply the claim.
\end{proof}

\section{Proof of the global estimates}
In this section we prove Theorems \ref{main theorem}, \ref{corollary theorem} and \ref{secder}. Central to our argument will be the following result of Rivi\`ere-Laurain \cite{LR}. This is not stated as a separate result in their paper however it can be found as the last estimate in the proof of Theorem 0.2 in their paper (assuming $f\equiv 0$ but the general case follows from standard elliptic theory). 

\begin{theorem}[Laurain-Rivi\`ere]\label{lau-riv}
There exists $\eps>0$ such that for all $\lambda,r,R>0$ satisfying $2r<R$, $\lambda<1$ and $\Om\in L^2(B_R\sm B_r , so(n)\otimes \bigwedge^1 \R^2)$, $f\in L\log L(B_R\sm B_r)$, $u\in W^{1,2}(B_R\sm B_r,\R^n)$ with 
$$-\Dl u = \Om \cdot \D u + f \,\,\,\,\,\,\text{and}\,\,\,\,\,\,\sup_{r<\rho<\fr{R}{2}} \int_{B_{2\rho}\sm B_{\rho}} |\Om|^2 \leq \eps,$$
there exists some $C=C(\lambda,n)<\infty$ such that 
$$\left\| \fr{1}{\rho}\pl{u}{\theta}\right\|_{L^{2,1}(B_{\lambda R}\sm B_{\fr{r}{\lambda}})} \leq C\left( 1+\|\Om\|_{L^2(B_R\sm R_r)}\right)(\|\D u\|_{L^2(B_R\sm B_r)} + \|f\|_{L\log L(B_R\sm B_r)}).$$ 
\end{theorem}

Using this Theorem and the previous results from section 3, we are now in a position to prove Theorem \ref{secder}.
\begin{proof}[Proof of Theorem \ref{secder}:]
The estimate for the $L^{2,1}$-norm of $\nabla u$ follows directly by combining Theorem \ref{lau-riv} with Proposition \ref{hopfest} and we are left with: 
\begin{align*}
\| \D u \|_{L^{2,1}(B_{\lambda}\sm B_{r/\lambda})}\leq& C\|\Om\|_{L^2(B_1\sm B_r)}\Big(\|\D u\|_{L^2(B_1\sm B_r )}+\|f\|_{L\log L(B_1\sm B_r)}\Big)\\
&+C\||\phi|^{1/2}\|_{L^{2,1}(B_1\sm B_r)}.
\end{align*}

It remains to show the $L^1$-estimate for $\nabla^2 u$. We first assume that 
\[
\int_{B_1\sm B_r}|\Om|^2 \leq \eps
\]
and we extend $\Om$ by zero to all of $B_1$. It follows from Theorem I.4 in \cite{Riv1}, that for $\eps$ small enough there exists $A\in W^{1,2}\cap L^\infty(B_1,GL(n))$ so that
\[
\text{div} (\nabla A-A \Om)=0
\]
and 
\[
\int_{B_1} |\nabla A|^2 \,\ dx+\text{dist} (A,SO(n))+\text{dist} (A^{-1},SO(n))\le C\int_{B_1\sm B_r}|\Om|^2.
\]
Moreover, there exists $B\in W^{1,2}(B_1,M(n))$ so that
\[
\nabla A - A \Omega =\nabla^\perp B
\]
and 
\[
\| \nabla B \|^2_{L^2(B_1)} \le C\int_{B_1\sm B_r}|\Om|^2.
\]
%Letting $\sqrt{r} <\lambda < 3/4$ and extending $u-u_{B_\lambda\sm B_{r/\lambda}}$ to all of $B_\lambda$, which we call $\tilde u$, so that
%\[
%\|\nabla \tilde u\|_{L^{2,1}(B_\lambda)} \le C \|\nabla u\|_{L^{2,1}(B_\lambda\sm B_{r/\lambda})}. 
%\]
%Here we use the corresponding statement in all $L^p$-spaces and the interpolation theorem of Marcinkiewicz.

Next, we extend $u-\dashint_{B_1\sm B_r} u$ to $\ti{u}:B_1\to \R^n$ which satisfies 
$$\|\D\ti{u}\|_{L^{2}(B_1)}\leq C\|\D u\|_{L^2(B_1\sm B_r)}$$ and $\D \ti{u} = \D u$ in $B_1\sm B_r$. 

Consider the Hodge decomposition of $A\D\ti{u}$ by $C\in W^{1,2}_0(B_1)$, $D\in W^{1,2}(B_1)$ of
\[
A\nabla  \ti{u} = \nabla C +\nabla^\perp D
\]
with 
\[
\| \nabla C \|_{L^{2}(B_1)}+\| \nabla D \|_{L^{2}(B_1)}=  \| \nabla \ti{u} \|_{L^{2}(B_1)}\leq C\|\D u\|_{L^2(B_1\sm B_r)}.
\]
The $L^1$-estimate for $\nabla^2 D$ follows since we have on $B_1$
\[
\Delta D = \nabla A \cdot \nabla^\perp  \ti{u}.
\]
Writing $D=h+\varphi$ with $h$ harmonic in $B_1$ with $h=D$ on $\partial B_1$ and $\varphi \in W^{1,2}_0(B_1)$ satisfies
\[
\Delta \varphi = \nabla A \cdot \nabla^\perp  \ti{u},
\]
we get from the results in \cite{CLMS}
\begin{align*}
\|  \varphi \|_{W^{2,1}(B_1)}%+\|\nabla \varphi\|_{L^2(B_\lambda)}
\le& C\| \nabla A\|_{L^2(B_1)} \| \nabla u\|_{L^2(B_1\sm B_r)}.
\end{align*}
Moreover, we get from standard estimates for harmonic functions and the fact that
\[
\| \nabla h \|_{L^2(B_1)}\le \| \nabla D\|_{L^2(B_1)},
\]
which follows since $h$ is harmonic and agrees with $D$ on the boundary, the estimate
\begin{align*}
\| \nabla^2 h\|_{L^1(B_\lambda)}\le& C(\lambda) \| \nabla h\|_{L^2(B_1)}\\
\le& C(\lambda)\| \nabla D\|_{L^2(B_1)} \\
\le& C(\lambda) \| \nabla u\|_{L^{2}(B_1\sm B_r)}%[1+\|\Om\|_{L^2(B_1\sm B_r)}]
\end{align*}
and therefore 
\[
\| \nabla D\|_{W^{1,1}(B_{\lambda})}\le C(\lambda)\| \nabla u \|_{L^{2}(B_1\sm B_r)}[1+\|\Om\|_{L^2(B_1\sm B_r)}].
\]
By the estimate for $D$ and the $L^{2,1}$-bound for $\nabla u$ we see that for any $r^{1/2} <\lambda < 1/2$ we have
\begin{align*}
\|\D C\|_{L^{2,1}(B_{\lambda}\sm B_{r\lambda^{-1}})} \leq& C(1+\|\Om\|_{L^2(B_1\sm B_r)})(\|\D u\|_{L^2(B_1 )} +\|f\|_{L\log L(B_1\sm \B_r)})\\
&+C\||\phi|^{1/2}\|_{L^{2,1}(B_1\sm B_r)}
\end{align*}

Moreover the function $C$ solves on $B_1 \sm B_{r}$
\[
\Delta C= \nabla^\perp B \cdot \nabla  \ti{u}-Af
\]
We decompose $C=\psi_1+\psi_2+v$, where $\psi_{1,2}\in W^{1,2}_0(B_1\sm B_{r})$, $\Delta \psi_1 =\nabla^\perp B \cdot \nabla  \ti{u}$, $\Delta \psi_2=-Af$ and $v$ is harmonic in $B_1 \sm B_{r}$ with $v=0$ on $\partial B_1$, $v=C$ on $\partial B_{r}$. Applying Lemma 2.1 in\cite{LR}, together with Lemma \ref{wente}, we get for every $\lambda$ as above
\[
\|  \D^2 \psi_1 \|_{L^1(B_{\lambda}\sm B_{r\lambda^{-1}})}+\|  \D \psi_1 \|_{L^{2,1}(B_{\lambda}\sm B_{r\lambda^{-1}})}% \sm B_{r/\lambda^2})}+\| \nabla^2 \phi \|_{L^1(B_{\lambda^2}\sm B_{r/ \lambda^2})} 
\le C \| \nabla B\|_{L^2(B_1)}\| \nabla \ti{u}\|_{L^2(B_1)}.
\]
Moreover, using standard elliptic theory (by extending $f$ and $\psi_2$ to zero in $B_r$) we get
\[
\|  \D^2 \psi_2 \|_{L^1(B_{\lambda}\sm B_{r\lambda^{-1}})}+\|  \D \psi_2 \|_{L^{2,1}(B_{\lambda}\sm B_{r\lambda^{-1}})}% \sm B_{r/\lambda^2})}+\| \nabla^2 \phi \|_{L^1(B_{\lambda^2}\sm B_{r/ \lambda^2})} 
\le C \| f\|_{L\log L(B_1\sm B_r)}.
\]
In particular we also conclude that
\begin{align*}
\|\nabla v \|_{L^{2,1}(B_{\lambda} \sm B_{r\lambda^{-1}})} \le& C (\| \nabla C \|_{L^{2,1}(B_{\lambda}\sm B_{r\lambda^{-1}})}+\| \nabla \psi_1 \|_{L^{2,1}(B_{\lambda}\sm B_{r\lambda^{-1}})}\\
&+\| \nabla \psi_2 \|_{L^{2,1}(B_{\lambda}\sm B_{r\lambda^{-1}})})\\
\le& C(\lambda)([1+\| \Om \|_{L^2(B_1 \sm B_r)}](\|\D u\|_{L^2(B_1\sm B_r )}+ \|f\|_{L\log L(B_1\sm B_r)})\\
&+\||\phi|^{1/2}\|_{L^{2,1}(B_1\sm B_r)}).
\end{align*}
Hence we conclude from Proposition \ref{harm} that 
\begin{align*}
\| \nabla^2 v\|_{L^1(B_{\lambda^2} \sm B_{r\lambda^{-2}})}\le& C(\lambda)([1+\| \Om \|_{L^2(B_1 \sm B_r)}](\|\D u\|_{L^2(B_1\sm B_r )}+ \|f\|_{L\log L(B_1\sm B_r)})\\
&+\||\phi|^{1/2}\|_{L^{2,1}(B_1\sm B_r)}).
\end{align*}
Combining the estimates for $\psi_1$, $\psi_2$ and $v$, we obtain
\begin{align*}
\| \nabla^2 C\|_{L^1(B_{\lambda^2} \sm B_{r\lambda^{-2}})}\le& C(\lambda)([1+\| \Om \|_{L^2(B_1 \sm B_r)}](\|\D u\|_{L^2(B_1\sm B_r )}+ \|f\|_{L\log L(B_1\sm B_r)})\\
&+\||\phi|^{1/2}\|_{L^{2,1}(B_1\sm B_r)}).
\end{align*}
Combining now the estimates for the $L^1$-norms of $\nabla^2 C$ and $\nabla^2 D$ with the formula 
\[
\nabla C+\nabla^\perp D= A\nabla \ti{u}=A\D u 
\]
in $B_\lambda \sm B_{r/\lambda}$ and the $L^{2,1}$-bound for $\nabla u$, we get the desired $L^1$-estimate for $\nabla^2 u$.

In the general case we use the same covering argument as Laurain-Rivi\`ere in their proof of Theorem 0.2 in \cite{LR} in order to reduce the general case to the previously considered one. 
\end{proof}
Now we are in a position to prove Theorem \ref{main theorem} by combining the uniform estimate on annuli with the bubbling argument from the appendix.
\begin{proof}[Proof of Theorem \ref{main theorem}:]
We argue by contradiction and we assume that there is a sequence $\{u_k\}$, $\{\Om_k\}$, $\{f_k\}$ as in the Theorem with
\[
\|\D u_k\|_{L^2(B_1)} +\|\Om_k\|_{L^2(B_1)}+\|f_k\|_{L\log L(B_1)} + \||\phi_k|^{\fr12}\|_{L^{2,1}(K)}\leq\Lambda<\infty
\]
and 
\[
\| \D^2 u_k\|_{L^{1}(K)}+\| \nabla u_k \|_{L^{2,1}(K)} \to \infty.
\] 
From a standard bubbling argument (see e.g. the appendix) it follows that one can decompose $B_1$ into a collection of bubble, neck and body regions for which we can apply $\eps$-regularity theory (see Theorem 1.6 in \cite{ST}), our improved neck results (see Theorem \ref{secder}) and simple covering arguments to conclude a global estimate for 
\[
\| \D^2 u_k\|_{L^{1}(K)}+\| \nabla u_k \|_{L^{2,1}(K)} ,
\] 
which contradicts the above.

More precisely, the bubbling argument decomposes $B_1$ into regions where bubbles form, regions where we have locally uniformly small $L^2$-norm of $\|\Omega_k\|_{L^2}$ and intermediate annular regions, on which we have smallness of $\|\Omega_k\|_{L^2}$ on each dyadic sub-annlus. On the first two regions it follows from Theorem 1.6 in \cite{ST} that the $W^{2,1}$-norm of $u_k$ is locally uniformly bounded, and on the intermediate annuli Theorem \ref{secder} yields the same conclusion. This yields the desired contradiction.
\end{proof}

In the following Proposition we find a condition under which the square root of the Hopf differential is bounded in $L^{2,1}$.
\begin{prop}\label{omegaperp}
There exists $C>0$ such that for all $\lambda\le \frac12$, $\Om\in L^2(B_1, so(n)\otimes \bigwedge^1 \R^2)$, $f,g\in L^2 (B_1,\R^n)$ and $u\in W^{1,2}(B_1,\R^n)$ solving
\begin{align*}
-\Dl u =& \Om \cdot \D u + f\ \ \ \text{and} \\
0=& \Om \cdot \nabla^\perp u+g,
\end{align*}
we have 
\begin{eqnarray*}
\left\| |\phi|^{1/2} \right\|_{L^{2,1}(B_{\lambda })} \leq  C\lambda^{1/2} \|\D u\|_{L^2(B_1)}(\|\D u\|_{L^2(B_1 )}+\|f\|_{L^2(B_1)}+\|g\|_{L^2(B_1)}) .
 \end{eqnarray*}
\end{prop}
\begin{proof}
We calculate
    \begin{align*}
\db \phi =& 2 \la \Delta u, u_z \ra =-2  \la \Omega_x u_x +\Omega_y u_y, u_x-i u_y \ra -2\la f, u_z \ra\\
=& -2 \la \Omega_y u_y, u_x \ra +2i \la \Omega_x u_x, u_y \ra -2\la f, u_z \ra\\
=& -2\la f+g, u_z \ra ,
\end{align*}
where we used that
\[
\Omega \cdot \nabla^\perp u= -\Omega_x u_y + \Omega_y u_x=-g.
\]
In particular, we conclude that
\[
\| \db \phi\|_{L^1(B_1)} \le c \| \nabla u\|_{L^2(B_1)}(\|f\|_{L^2(B_1)} +\| g\|_{L^2(B_1)})
\]
and hence it follows from elliptic regularity theory that 
\[
\| \phi\|_{L^{2,\infty}(B_{1/2})} \le c\| \nabla u\|_{L^2(B_1)}(\|\nabla u\|_{L^2(B_1)}+\|f\|_{L^2(B_1)} +\| g\|_{L^2(B_1)}).
\]
From this we get that for every $\lambda \le \frac12$ we have the estimate
\begin{align*}
\| |\phi|^{1/2} \|_{L^{2,1}(B_\lambda)} \le& c\lambda^{1/2} \| |\phi|^{1/2}\|_{L^{4,\infty}(B_\lambda)}\le c\lambda^{1/2} \| \phi\|_{L^{2,\infty}(B_\lambda)}\\
\le& c\lambda^{1/2} \| \nabla u\|_{L^2(B_1)}(\|\nabla u\|_{L^2(B_1)}+\|f\|_{L^2(B_1)} +\| g\|_{L^2(B_1)})
\end{align*}
and this finishes the proof of the Proposition.
\end{proof}
\begin{proof}[Proof of Theorem \ref{corollary theorem}:]
The result follows from combining Theorem \ref{main theorem} and Proposition \ref{omegaperp}.
\end{proof}
It is of course a natural question to ask if there are examples of solutions of the system
\begin{align*}
-\Dl u =& \Om \cdot \D u + f\ \ \ \text{and} \\
0=& \Om \cdot \nabla^\perp u+g,
\end{align*}
where $\Om$, $f$ and $g$ are as in the above Proposition. In the following Lemma we show that every critical point of a conformally invariant Lagrangian with quadratic growth in the gradient is a solution of this system with $f=g=0$. In particular, important examples are harmonic maps and surfaces of prescribed mean curvature $H\in L^\infty$.
\begin{lemma}\label{crit}
Let $N^k\subset \R^n$ be a $C^2$ submanifold and let $\omega$ be a $C^1$ two-form on $N$ with bounded $L^\infty$-norm of $d\omega$. Then every critical point $u\in W^{1,2}(B_1,N)$ of
\begin{align}
F(u)=\frac12 \int_{B_1} \left( |\nabla u|^2 +u^\star \omega \right) dx \label{grut}
\end{align}
satisfies 
\[
-\Dl u =\Omega(u) \cdot \nabla u
\]
with 
\[
\Omega^i_{j}(u)=\left(A^i_{jk}(u)-A^j_{ik}(u) \right) \nabla u^k+\frac14 \left(\lambda^i_{jk}(u) -\lambda^j_{ik}(u) \right) \nabla^\perp u^k,
\]
where $A,\lambda \in C^0(B_1,M_n \otimes \bigwedge^1 \R^2)$ satisfy for every $i,k\in \{ 1,\ldots,n\}$ 
\begin{align}
\sum_j A^j_{ik} (u)\nabla u^j=0 \label{ortho} 
\end{align}
and for all $i,j,k \in \{1,\ldots,n\}$ we have
\[
A^i_{jk}(u)=A^i_{kj}(u),\,\ \lambda^i_{jk}(u)=-\lambda^i_{kj}(u)\ \ \ \text{and} \,\ \lambda^i_{jk}(u)=-\lambda^j_{ik}(u).
\]
Moreover, we have
\[
\Om(u) \cdot \nabla^\perp u=0.
\]
\end{lemma}
\begin{proof}
All statements except the last one can be found in Theorem I.2 of \cite{Riv1}. In order to show the last claim we note that
\begin{align*}
\Om^i_{j}(u) \nabla^\perp u^j=& A^i_{jk}(u) \left(-\partial_x u^k \partial_y u^j+\partial_y u^k \partial_x u^j \right) \\
&-A^j_{ik}(u)  \left(-\partial_x u^k \partial_y u^j+\partial_y u^k \partial_x u^j \right)\\
&+ \frac12 \lambda^i_{jk}(u) \left(\partial_y u^k \partial_y u^j+\partial_x u^k \partial_x u^j \right)
\end{align*}
and this expression vanishes, since  the first term vanishes using the symmetry of $A^i_{jk}(u)$ and the antisymmetry with respect to $j$ and $k$ of the term in brackets. The same argument applies to the third term, and the second one vanishes because of \eqref{ortho}.
\end{proof}
\begin{remark}
\begin{itemize}
\item[1)] In the proof of \cite[Theorem 4.2]{LR} the crucial observation was that $\Om\cdot \D u$ is orthogonal to $\D u$, or equivalently that the Hopf differential is holomorphic, which was used in conjunction with the Pohozaev identity to convert $L^2$ angular control into $L^2$ radial control. We point out that $\Om\cdot \Db u = 0$ implies the former, moreover as stated earlier it can be understood as the ``curl" counterpart to the divergence type equation \eqref{omega} - see Remark \ref{dbaretc}.  
\item[2)] It was shown by Gr\"uter \cite{grueter1} that every conformally invariant variational integral with quadratic growth in the gradient can be written in the form \eqref{grut}.
\end{itemize}
\end{remark}

\section{Energy identities}
In this section we show that using the estimates we derived earlier, one can deduce various energy identities for sequences of solutions of the systems under consideration. These results improve the known energy identity results since we even get that no energy is lost for the gradients in the $L^{2,1}$-norm. 
\begin{proof}[Proof of Theorem \ref{EIL21}:]
It follows from Theorem 1.6 in \cite{ST} that the sequence $u_k$ is uniformly bounded in $W^{2,1}$ locally away from at most finitely many points $S=\{ x_i | i\in I\}$ at which the $L^2$-norm of $\Omega_k$ concentrates. Therefore, using Theorem 1.2 in \cite{ST}, $u_k$ converges weakly in $W^{1,2}(B_1,\R^n)$ and strongly in $W^{1,2}_{loc}(B_1\sm S,\R^n)$ to a limit $u$ which is a weak solution of the limiting system, as described in the statement of the Theorem. Here we use the strong convergence in $W^{1,2}$ away from $S$ and the fact that $S$ consists of at most fintely many points.
It follows again from Theorem $1.6$ in \cite{ST} that $u\in W^{2,1}_{loc}(B_1,\R^n)$.

Using the bubbling argument from the appendix, we see that we can reduce the analysis to the situation of one bubble forming at the origin. Away from the origin we have strong $W^{1,2}$ convergence of $u_k$ to $u$ and therefore we conclude the strong convergence of the gradients in the $L^2$-norm.

Now it follows from the definiton of a bubble which forms at the origin, that there exist sequences $x_k\to 0$ (wlog we assume $x_k\equiv 0$) and $r_k \to 0$ so that 
\[
\sup_{B_{r_k}(y) \subset B_{1/2}} \| \Om_k\|_{L^2(B_{r_k}(y))}=  \| \Om_k\|_{L^2(B_{r_k}(0))}=\eta_0 ,
\]
where $\eta_0$ is as in Theorem $1.2$ of \cite{ST}. Next we rescale the maps
\begin{align*}
\tilde u_k (x) =& u_k (r_kx),\\
\tilde \Om_k(x) =& r_k \Om (r_kx) \ \ \ \text{and}\\
\tilde f_k(x)=& r_k^2 f_k(r_k x)
\end{align*}
and it follows from Theorem 1.2 in \cite{ST} that $\tilde u_k$ converges strongly in $W^{1,2}(\R^2,\R^n)$ to a solution $\omega\in W^{2,1}_{loc}(\R^2,\R^n)$ of 
\begin{align*}
-\Dl \omega =& \tilde \Om \cdot \D \omega, 
\end{align*}
where $\tilde \Om \in L^2_{loc}(\R^2,\R^n)$ is the weak limit of $\tilde \Om_k$. %Moreover, we conclude from the bubbling argument in the appendix that $\tilde \Om \cdot \nabla^\perp \omega=0$.

In particular, this shows that the $L^2$-norm of $\D \ti{u}_k$ converges to the $L^2$-norm of $\D \omega$ in $B_{R}$, for every $R>0$.

In order to finish the proof of the first part of the Theorem, it therefore remains to show that the $L^2$-norm of $\nabla u_k$ on the annulus $A_{\delta,R,k}:=B_\delta \sm B_{Rr_k}$ converges to zero, as $k\to \infty$, $\delta \to 0$ and $R\to \infty$. 

From the bubbling argument, combined with Theorem \ref{secder}, and the estimate for the $L^{2,\infty}$-norm of $\nabla u_k$ from Lemma 3.1 in \cite{LR} or from Lemma 4.2 in \cite{LiZhu}, together with the estimate from Theorem 1.6 in \cite{ST} (which is used in order to show that the $L^2$-norm of $\nabla u_k$ converges to zero on dyadic annuli), we conclude that the $L^2$-norm of $\nabla u_k$ converges to zero on the annulus, as $k\to \infty$, $\delta \to 0$ and $R\to \infty$.

In order to prove the energy identity for the $L^{2,1}$-norms and the no-neck property, we note that the additional assumption $f_k\in L^p$, with uniform bounds, improves the above convergence results of $u_k$ to the weak limit $u$, and of $\tilde u_k$ to the bubble $\omega$ to local strong convergence in $W_{loc}^{1,q}(B_1\backslash S, \R^n)$ resp. $W_{loc}^{1,q}(\R^2, \R^n)$ by using Theorem 1.1 in \cite{ST}. Hence it follows that the $L^{2,1}$-norms of $\D u_k$ resp. $\D \tilde u_k$ converge to the corresponding quantities of $\D u$ resp. $\D \omega$ away from the origin resp. in $B_R$ for every $R>0$.
Next, using the fact that the $L^2$-norm of $\nabla u_k$ and the $L^{2,1}$-norm of $|\phi_k|^{1/2}$ converge to zero on the annulus, another application of Theorem  \ref{secder} then yields that also the $L^{2,1}$-norm of $\nabla u_k$ converges to zero on the annulus and this finishes the proof of the second energy identity \eqref{energyequality10}.

Finally, in order to show the no-neck property, we note that one can extend the map \[
u_k-\mint_{A_{\delta/2,2R,k}} u_k:A_{\delta/2,2R,k}\to \R^n
\] 
to a map $\tilde u_k$ with compact support in $A_{\delta,R,k}$, so that for every $p\in (1,\infty)$
\[
\|\D \tilde u_k \|_{L^p(\R^2)} \le C \| \D u_k\|_{L^p(A_{\delta/2,2R,k})},
\]
where $C$ doesn't depend on $\delta$, $R$ and $k$. From the Marcinkiewicz interpolation theorem we then conclude the estimate
\[
\|\D \tilde u_k \|_{L^{2,1}(\R^2)} \le C \| \D u_k\|_{L^{2,1}(A_{\delta/2,2R,k})}.
\]
Next we use the Sobolev embedding theorem (see Theorem 3.3.4 in \cite{helein02}) to conclude
\begin{align*}
\text{osc}_{A_{\delta/4,4R,k}} u_k =&  \text{osc}_{A_{\delta/4,4R,k}} (u_k -\mint_{A_{\delta/2,2R,k}} u_k)\\ 
\le& 2\|\tilde u_k \|_{L^\infty(B_{\delta/4})}\\
\le& C\|\D \tilde u_k \|_{L^{2,1}(\R^2)} \\
\le& C \| \D u_k\|_{L^{2,1}(A_{\delta/2,2R,k})} \to 0,
\end{align*}
which finishes the proof of the no neck property. 

The fact that the bubbles are weakly conformal follows again by the convergence to zero of the Hopf differential on bubble domains. 
\end{proof}
\begin{remark}\label{remark}
It follows from Proposition \ref{omegaperp} that Theorem \ref{EIL21} applies in particular to sequences of solutions $u_k$ of the system
\begin{align*}
-\Dl u_k =& \Om_k \cdot \D u_k + f_k\ \ \ \text{and} \\
0=& \Om_k \cdot \nabla^\perp u_k+g_k
\end{align*}
under the assumption that there exists a constant $\Lambda>0$ so that for every $k\in \N$
\[
\int_{B_1} \left( |\D u_k|^2+|\Om_k|^2+|f_k|^2 +|g_k|^2 \right)dx \le \Lambda.
\]
In particular, using Lemma \ref{crit}, we get new energy identites for sequences of approximate harmonic maps with tension fields bounded in $L^2$ and for sequences of surfaces with prescribed mean curvatures in $L^\infty$.
 
We also remark that the $L^2$-energy identity \eqref{energyequality9} was previously shown in Theorem 0.3 of \cite{LR} for sequences of critical points of every conformally invariant Lagrangian with quadratic growth in the gradient. 
\end{remark}
Imposing additional growth conditions on $\Om_k$ and its first derivative, which are for example satisfied by sequences of approximate harmonic maps with tension fields bounded in $W^{1,2}$ and by sequences of immersions with prescribed mean curvature in $W^{1,\infty}$, we can additionally conclude an energy identity for the second derivative in $L^1$.
\begin{theorem}\label{EIW21}
Let $u_k\in W^{1,2}(B_1,\R^n)$ be a sequence of solutions of 
\begin{align*}
-\Dl u_k =& \Om_k \cdot \D u_k + f_k,
\end{align*}
where $f_k\in W^{1,2}(B_1,\R^n)$ and we assume that there exists a constant $\Lambda>0$ so that for every $k\in \N$
\[
\int_{B_1} \left( |\D u_k|^2+|\Om_k|^2+|f_k|^2+|\nabla f_k|^2 \right)dx+ \| |\phi_k|^{1/2}\|_{L^{2,1}(B_1)} \le \Lambda.
\]
Moreover, we assume that for every $x\in B_1$, 
\begin{align*}
|\Omega(x)|\le& C|\nabla u_k(x)|,\\
|\nabla \Omega(x)|\le& C(|\nabla u_k(x)|^2 +|\nabla^2 u_k(x)|).
\end{align*}
and 
\[
\| |\phi_k|^{1/2}\|_{L^{2,1}(Z)} \to 0
\]
for every subset $Z\subset B_1$ with $|Z| \to 0$. Then there exists a subsequence, still denoted by $u_k$, $\Om_k$ and $f_k$, so that $u_k \rightharpoonup u \in W^{3,p}(B_1,\R^n)$, for all $1\le p<2$, weakly in $W^{1,2}(B_1,\R^n)$, $\Omega_k \rightharpoonup \Om$ and $f_k \rightharpoonup f$ weakly in $W^{1,2}(B_1,\R^n)$ resp. $L^2(B_1,\R^n)$ and the limits are solutions of 
\begin{align*}
-\Dl u =& \Om \cdot \D u + f.
\end{align*}
Moreover there exist at most finitely many $\omega$-bubbles $\omega^{i,j}:\R^2\to \R^n$, $1\leq i \leq p$, $1\leq j\leq j_i$, i.e. solutions of
\[
-\Delta \omega^{i,j} = \Om^{i,j} \cdot \D \omega^{i,j},%,\ \ \ \Om^{i,j} \cdot \nabla^\perp \omega^{i,j}=0,
\]
which are conformal (as in Thereom \ref{EIL21}), sequences of points $x_k^{i,j} \in B_1$, $x_k^{i,j} \to x_i$, and sequences of radii $t_k^{i,j}\in \R_+$, $t_k^{i,j} \to 0$ as in Theorem \ref{EIL21}, such that for every $r<1$ with $\{x_1,\ldots,x_p\} \in B_r$ we additionally have
\begin{align}
\lim_{k\to \infty} \|\nabla^2 u_{k}\|_{L^{1}(B_r,\R^n)}&=\| \nabla^2 u\|_{L^{1}(B_r,\R^n)}+\sum_{i=1}^p\sum_{j=1}^{j_i} \|\nabla^2 \omega^{i,j}\|_{L^{1}(\R^2,\R^n)}.  \label{energyequality10a}
\end{align}
\end{theorem}
\begin{proof}
Note that under the assumptions on $\Omega$ it follows first from Theorem 1.1 in \cite{ST} that $\nabla u_k\in L^q$ and then $\nabla^2 u_k \in L^q$, for every $1<q<\infty$, away from at most finitely many points with uniform bounds depending only on $\Lambda$. Differentiating the equation then yields a uniform bound for $\nabla^3 u_k\in L^p$, $1\le p<2$, away from at most fintely many points.
Once this higher regularity result has been obtained, we can then copy the proof from above Theorem \ref{EIL21} since we now have strong $W^{2,1}$-converge of $u_k$ to the weak limit $u$ away from the bubbles, and we also have strong $W^{2,1}$-convergence of $u_k$ to the bubbles $\omega^{i,j}$. In the intermediate annular regions we use Theorem \ref{secder} in order to conclude that the $W^{2,1}$-norm of $\nabla u_k$ converges to zero in these regions.
\end{proof}
\begin{remark}
As in Remark \ref{remark} we can replace the condition on the Hopf differential by imposing that $\Om_k\cdot \Db u_k = g_k$ is uniformly bounded in $L^2$ and we can make the same conclusion. In particular this applies to solutions of critical points of conformally invariant integrals, and especially sequences of harmonic mappings. 
\end{remark}

\section{The Equation on surfaces}
Here let $(\Sigma, h)$ denote a Riemann surface equipped with a metric $h$ and we consider a map $F\in W^{1,2}(\Sigma,\R^n)$ and $\Om\in L^2(\Sigma,so(n)\otimes\bigwedge^1 \T^{\ast}\Sigma)$ solving 
\begin{eqnarray*}
\Dl_h F = \Om\cdot_h \ed F \,\,\,\,\,\,\,\,\text{and}\,\,\,\,\,\,\,\,\,
\Om\wedge \ed F = 0, %&&\text{or $u$ conformal}.
\end{eqnarray*}
where $\Delta_h$ is the Laplace-Beltrami operator on $M$ and $\cdot_h$ denotes the multiplication of one-forms with respect to the metric $h$. If we take any isothermal coordinates $\psi:B_1\to U$ on $U\In\Sigma$ we may consider $F\in W^{1,2}(B_1,\R^n)$, $\Om\in L^2(B_1,so(n)\otimes \bigwedge^1 \T^{\ast}\R^2)$ solving 
\begin{eqnarray*}
&-\Dl F = \Dl_{\dl} F =\Om \cdot_{\dl} \ed F = \Om\cdot \ed F&\,\,\,\,\,\text{and}\\
&\Om\wedge \ed F = 0 &%&&\text{or $u$ conformal}
\end{eqnarray*}
where $h=e^{2w}\dl$ in these coordinates for some function $w:B_1\to \R$. It is not difficult to check that for any $V\In B_1$,  $\|\D F\|_{L^2(V,\dl)} = \|\D F\|_{L^2(\psi(V),h)}$, $\|\Om\|_{L^2(V,\dl)} = \|\Om\|_{L^2(\psi(V),h)}$ and therefore we use either of these norms without ambiguity. We also have some uniform constant $C>0$ such that 
$$\|Hess_h(F)\|_{L^1(\psi(V),h)} \leq C(\|\D^2 F\|_{L^1(V,\dl)} +\|\D F\|_{L^{2,1}(V,\dl)}\|\D w\|_{L^{2,\infty}(V.\dl)}).$$

Thus, using Theorem \ref{main theorem}, we may conclude the following: 

\begin{theorem}\label{global_surf}
Let $(\Sigma, h)$ denote a compact Riemann surface equipped with a metric $h$, and maps $F\in W^{1,2}(\Sigma,\R^n)$, $\Om\in L^2(\Sigma,so(n)\otimes\bigwedge^1 \T^{\ast}\Sigma)$ solving 
\begin{eqnarray*}
\Dl_h F = \Om\cdot_h \ed F \\
\Om\wedge \ed F = 0 %&&\text{or $u$ conformal}.
\end{eqnarray*}
Now suppose that there exists a finite atlas (with $K$ elements, say)  of isothermal coordinates over simply connected domains such that the conformal factors $w_i\in W^{1,(2,\infty)}(B_1)$. Then there exists some \[
C=C(\|\Om\|_{L^2(\Sigma)}, \|\D F\|_{L^2(\Sigma)},\max_i \|\D w_i\|_{L^{(2,\infty)}(B_1)}, K)<\infty
\] 
such that 
\[
\|Hess_h(F)\|_{L^{1}(\Sigma)} \leq C.
\]
\end{theorem}
Moreover, using Theorem \ref{EIL21}, it is easy to check that we have the following: 
\begin{theorem}
Let $(\Sigma, h_k)$ be a compact surface equipped with a sequence of smooth metrics. Let $F_k\in W^{1,2}(\Sigma,\R^n)$, $\Om_k\in L^2(\Sigma,so(n)\otimes\bigwedge^1 \T^{\ast}\Sigma_k)$ solve 
\begin{eqnarray*}
\Dl_{h_k} F_k = \Om_k\cdot_{h_k} \ed F_k \\
\Om_k\wedge \ed F_k = 0. %&&\text{or $u$ conformal}.
\end{eqnarray*}
Assuming that the metrics $h_k$ converge smoothly to some limit metric $h$ and that there exists a constant $\Lambda <\infty$ so that
$$\|\Om_k\|_{L^2(\Sigma,h_k)} + \|\D F_k\|_{L^2(\Sigma, h_k)} \leq \Lambda, $$ 
there exist limits $F_{\infty}:\Sigma \to \R^n$, $\Om_{\infty}\in L^2(\Sigma,so(n)\otimes \bigwedge^1 \T^{\ast}\Sigma)$ and a collection of $\Om^i$ bubbles $\om^i:(S^2,h_{round})\to \R^n$ with $\Om^i\in L^2 (S^2,so(n)\otimes \bigwedge^1\T^{\ast}S^2)$, such that: 
$$\Dl_h F_{\infty} = \Om_{\infty} \cdot_h \ed F_{\infty}, \,\,\,\,\,\,\,\Dl \om^i = \Om^i \cdot \ed \om^i,$$
$$\Om_{\infty}\wedge \ed F_{\infty} = 0,\,\,\,\,\,\,\,\,\,\,\,\Om^i\wedge \ed \om^i = 0$$
$$\lim_{k\to\infty} \|\D F_k\|_{L^2(\Sigma,h_k)}^2 = \|\D F_{\infty} \|_{L^2(\Sigma,h)}^2 + \sum_i \|\D \om^i \|_{L^2(S^2)}^2$$ and the collection $F_{\infty}(\Sigma)\cup_i \om^i(S^2)$ is connected without necks. 

\end{theorem}
\begin{remark}
We could of course have formulated these theorems involving $f,g$ as in the other parts of the paper, but for the sake of simplicity we stick to this case. 
\end{remark}
Examples of such solutions are harmonic maps from surfaces but one other area of potential interest is in the theory of $W^{2,2}$ conformal immersions - see e.g. \cite{Kuwert_Li}, \cite{LR2}. Given an arbitrary Riemann surface equipped with a metric of constant curvature $(\Sigma,h)$ (where $h=-1,0,1$ and $Area(\Sigma,h)=1$ when $h=0$) and a compact Riemannian manifold isometrically embedded $N\emb\R^n$ then we say that $F\in W^{2,2}_{conf}((\Sigma,h),N)$ if in any local conformal parameter the induced metric $g_{ij} = e^{2w}h_{ij}=e^{2\hat{w}}\dl_{ij}$ has $w\in L^{\infty}_{loc}$ and $F\in W^{2,2}(\Sigma,N)$ is an immersion, and finally 
$$\int_{\Sigma} |\A|^2 \ed V_g \leq C<\infty$$where $\A$ is the second fundamental form of $F(\Sigma)\In N$.% then we automatically get $W^{1,(2,\infty)}$ control on any $w$ via the weak Liouville equation: 

%Fix a smooth conformal background metric $h$, which gives the existence of some $\hat{w}\in L^{\infty}_{loc}(\Sigma)$ with $g=F^{\ast}\dl = e^{2\hat{w}}h$. By the Liouville equation we have 
%$$\Dl_h \hat{w} = K_{g}e^{2\hat{w}} - K_h \,\,\,\,\,\,\,\,\text{thus}\,\,\,\,\,\,\,\,\,\,\,\,|\Dl_h \hat{w}| \leq C<\infty$$ and we conclude from elliptic regularity theory that
%$$\|\D \hat{w}\|_{L^{2,\infty}(\Sigma)} \leq C.$$ Notice that in any isothermal coordinates $\hat{w} = w - v$ where $h=e^{2v}\dl$ in these coordinates - thus giving the global control on $\D w$ as above.  

If we let $H:\Sigma\to \R^m$ denote the mean curvature vector of $F$ in $N$ then we have 
$$\Dl_g F = 2H + A(F)(\D_g F, \D_g F)$$ where $A$ is the second fundamental form of $N\emb \R^n$, so that on each conformal chart we have 
$$-\Dl F = 2H e^{2\hat{w}} + A(F)(\D F, \D F) = H |\D F|^2+ A(F)(\D F, \D F).$$ 
Following \cite{sharp} if we let $\Om^i_j : = H^i \ed F^j - H^j \ed F^i + (A^i_{jk}(F) - A^j_{ik}(F))\ed F^k$ then we are in the situation as above for $u=F$, $\|\Om\|_{L^2(\Sigma)}\leq C( \|H\|_{L^2(\Sigma)}+ \|\D F\|_{L^2(\Sigma)})$ and $\|\D F\|^2_{L^2(\Sigma)} = Area (F(\Sigma)).$ 

We note that in a similar setting (and even for sequences of possibly degenerating surfaces), an energy identity and no-neck result has been obtained previously by Chen and Li \cite{chenli}. 

Here we recover the following which uses the results \cite[Theorem 3.1]{LR2}, \cite[Lemma II.2]{RM}, and the results from this paper. 
\begin{theorem}
Let $F_k\in W^{2,2}_{conf}((\Sigma_k,h_k),N)$ be such that
$$\int_{\Sigma_k} |\A_k|^2 \ed V_{g_k} + Area(F_k(\Sigma_k)) \leq \Lambda<\infty.$$ Then there exists some $C=C(\Lambda,N)$ such that 
$$\|Hess_{h_k}(F_k)\|_{L^1(\Sigma_k)}\leq C$$ moreover we get a bubble tree, energy identity, no necks result and a limiting branched conformal immersion on a stratified Riemann surface. 
\end{theorem}

We end this section with the following remark: we could consider simply $F\in W^{1,2}(\Sigma,\R^n)$ that is weakly conformal with respect to $h$. We would say that $F$ has mean curvature in $L^2$ if there exists some function $H\in L^2(\Sigma,\R^n)$, $\la H, \D F\ra = 0$ almost everywhere, and 
$$\tau_g(F)=\Dl_g F - A(F)(\D_g F, \D_g F) = 2H$$ weakly. Now letting $\Om$ be as above we are back in the situation of Theorem \ref{global_surf}. In this second case we say that $F$ is a weak conformal immersion with bounded Willmore energy - an open question here is whether the zero set of $\D F$ consists of isolated points? Or can we say that $F\in W^{1,\infty}$? If both of these are true then $F$ would be a branched $W^{2,2}$ conformal immersion.

%With these examples in mind we also point out a generalisation of the results of Miaomiao Zhu for harmonic maps from degenerating Riemann surfaces \cite{Zhu_M}. First of all, the estimates we obtain along degenerating annuli are equivalent to those estimates one would obtain for the same equation on a long Euclidean cylinder of fixed unit radius. Therefore we can recover the results from \cite[Section 3]{Zhu_M} except this time for solutions $u$ to 
%$$-\Dl u  = \Om \cdot \ed u \,\,\,\,\,\,\text\{and}\,\,\,\,\,\,\,\,\,\Om\wedge \ed u  = 0 $$
%which we know implies that the Hopf differential is holomorphic. Therefore, using the results we obtain here 

\appendix

\section{The bubbling argument}

We present a bubbling argument in order to decompose $B_1$ into a collection of bubble, neck and body regions which are defined below. This decomposition, together with Theorem \ref{secder}, is the key ingredient in the proof of the global energy estimate and the enrgy identities. 
Similar constructions can be found in \cite{DingTian}, \cite{Parker}, \cite{LR} and \cite{bernardriviere}.

Here we are interested in a sequence of solutions $u_k\in W^{1,2}(B_1, \R^n)$ of
\[
-\Delta u_k = \Omega_k \cdot \nabla u_k +f_k,
\] 
where $\Omega_k \in L^2(B_1, so(n) \otimes \bigwedge^1 \R^2)$, $f_k \in L\log L(B_1,\R^n)$ with 
$$\|\D u_k\|_{L^2(B_1)} +\|\Om_k\|_{L^2(B_1)}+\|f_k\|_{L\log L(B_1)} + \||\phi_k|^{\fr12}\|_{L^{2,1}(B_{1})}\leq\Lambda<\infty.$$

By the regularity results in \cite{ST} (see Theorem 1.6) we know that there is some uniform $\eps >0$ such that if the objects above solve the PDE on any domain $U\In\R^2$ whenever 
$$\rho_k := \inf\{\rho>0 | \sup_{B_{\rho}(x)\In U} \|\Om_k \|_{L^2(B_{\rho}(x))} = \eps\} \geq \al >0,$$
independently of $k$, then we get $W^{2,1}$ estimates for $u_k$ locally on $U$, and thus $u_k\to u\in W^{1,2}_{loc}(U)$ strongly, where $u$ is a solution of the limit equation
\[
-\Delta u =\Omega \cdot \D u+f
\] 
as described in the proof of Theorem \ref{EIL21}.

%Also there can be at most $N \leq \fr{\Lambda}{\eps}$ points $x^i\in B_{\fr12}$ such that 
%$$\liminf_{r\to 0} \|\Om_k\|_{B_r(x^i)} > \eps.$$
%Moreover for any $R>0$ (to be fixed later), for any $x\in B_{\fr12}\sm \cup_{i=1}^N B_R(x^i)$ there is some uniform $r_0 >0$ with $\|\Om_k\|_{L^2(B_{r_0}(x))} \leq \eps$. Thus, a covering argument yields a uniform bound on $\| u_k\|_{W^{2,p}(B_{\fr12}\sm \cup_{i=1}^N B_R(x^i))}$. 
%Therefore up to subsequence (by result of Rivi\`ere?) there exists some map $u\in W^{1,2}\cap W^{1,q}_{loc}(B_{\fr12}\sm \cup_{i=1}^N \{x^i\})$ for any $q<\infty$, $\Om \in...$, $f,g\in ...$ such that $u$, $\Om$, $f$, $g$ solve what we want and $u_k\to u$ in $W_{loc}^{1,q} (B_{\fr12}\sm \cup_{i=1}^N \{x^i\})$.
\paragraph{\bf{The basic set-up}} We use the following standard bubbling argument:  

Define $\rho_{k^1}$ and $x_{k^1}$ by 
$$\rho_{k^1} := \inf\{\rho>0 |\|\Om_k\|_{L^2(B_{\rho_{k^1}}(x_{k^1}))}= \sup_{B_{\rho}(x)\In B_1} \|\Om_k \|_{L^2(B_{\rho}(x))} = \eps\},$$ and let $U_{k^1} = B_{\rho_{k^1}}(x_{k^1})$.  

We assume $\rho_{k^1}\to 0$ - if not we have a global $W^{2,1}$ estimate and no bubbling phenomenon occurs. Moreover we pick a subsequence so that $x_{k^1}\to x^1\in \overline{B}_1$ and if $x^1\in \de B_1$ then we discard this sequence. 

Now we pick $\rho_{k^2}$ and $x_{k^2}$ such that 
 $$\rho_{k^2} := \inf\{\rho>0 |\|\Om_k\|_{L^2(B_{\rho_{k^2}}(x_{k^2})\sm U_{k^1})}= \sup_{B_{\rho}(x)\In B_1} \|\Om_k \|_{L^2(B_{\rho}(x)\sm U_{k^1})} = \eps\}$$
and set $U_{k^2} : = B_{\rho_{k^2}}(x_{k^2})\cup U_{k^1}$. 

First we check if $\rho_{k^2}\to 0$ - if not we stop here. Once again we pick a subsequence so that $x_{k^2}\to x^2\in \overline{B}_1$ and if $x^2\in \de B_1$ then we discard this sequence.

Now we ask whether we have 
$$S^{12}_k:= \left(\fr{\rho_{k^2}}{\rho_{k^1}} + \fr{|x_{k^1}-x_{k^2}|}{\rho_{k^2}}\right)\to \infty$$
or not. 

If $S_k^{12}$ remains bounded we discard this sequence - since this means that the energy on $B_{\rho_{k^2}}(x_{k^2})\sm B_{\rho_{k^1}}(x_{k^1})$ is contributing to the ``first bubble". 

If $S^{12}_k$ becomes unbounded we remember the sequence because it means either $x_{k^1}$ and $x_{k^2}$ converge to different points, or they converge to the same point at different scales, or they converge to the same point at the same scale but the scales go to zero so quickly (compared to their respective rate of convergence) that they remain conformally very far from each other!

%However, an important remark is that we keep track of the sets $U_{k^i}$ regardless of the behaviour of the sequence - thus ensuring that each stage of the process takes into account $\eps$ of $\Om$- energy. Thus since we have a global bound the process eventually stops.  

We now inductively continue to choose such scales and points $(\rho_{k^{i+1}}, x_{k^{i+1}})$ such that $$\rho_{k^{i+1}} := \inf\{\rho>0 |\|\Om_k\|_{L^2(B_{\rho_{k^{i+1}}}(x_{k^{i+1}})\sm U_{k^i})} = \sup_{B_{\rho}(x)\In B_1} \|\Om_k \|_{L^2(B_{\rho}(x)\sm U_{k^i})} = \eps\}$$ and $U_{k^{i+1}}:=B_{\rho_{k^{i+1}}}(x_{k^{i+1}})\cup_{j=1}^{i} U_{k^j} $.

Again, if $\rho_{k^{i+1}}$ does not converge to zero we stop.
If it does converge to zero, we take a convergent subsequence for the $x_{k^{i+1}}$ as above and also we check:
Does $S_k^{{l,i+1}}$ remain bounded for some $1\leq l\leq i$? 

If yes, forget the sequence. 

If no remember the sequence and carry on. Remember the construction of the $U_i$ is unaffected by this step - the only thing we do is decide whether to remember the sequence of scales and points - or not. 

This process eventually stops after finitely many iterations since each time we are taking away a fixed amount of energy and two such domains never overlap by construction. Let $Q$ denote the total number of distinct point-scale sequences and we have $Q\leq \fr{\Lambda}{\eps}$. 

We are left in the following situation: 
We have finitely many points $x^j$ and for each point we have a maximal set of  finitely many point-scale sequences $(\rho^j_{k^i}, x^j_{k^i})$ with $(S_j)_k^{il} \to \infty$ when $i<l$ and $x^j_{k^i}\to x^j$ for every $i$. Moreover at each scale we are accounting for a fixed amount of $\|\Om\|_{L^2}$ - we shall refer to these point-scale sequences as bubble sequences in the sequel as the below argument shows we end up with a bubble for each one. %Let $Q$ denote the total number of point-scale sequences (each of which is equivalent to a bubble).  

Notice also that if we let $\hat{u}^j_{k^i}(x) := u(x^j_{k^i} + \rho^j_{k^i}x)$ then this map is defined on larger and larger regions of $\R^2$%, moreover there are only finitely many points on the interior of $\R^2$ that have energy concentration: 

Setting $\hat{\Om}^j_{k^i}(x) := \rho^j_{k^i}\Om(x^j_{k^i} + \rho^j_{k^i}x)$ and $\hat{f}^j_{k^i}(x) := (\rho^j_{k^i})^2 f(x^j_{k^i} + \rho^j_{k^i}x)$ these objects solve 
$$-\Dl \hat{u}^j_{k^i} = \hat{\Om}^j_{k^i} \cdot \D \hat{u}^j_{k^i} + \hat{f}^j_{k^i}.$$ Thus by the scaling properties for the $L\log L$ - norm and also $\Om$ and $u$ (see e.g. \cite{ST}) we know that $\|\hat{\Om}^j_{k^i}\|_{L^2}$ only concentrates (if at all) at finitely many points $y^1,\dots y^l$ for $l<Q$; $\|\D \hat{u}^j_{k^i}\|_{L^2}$, $\|\hat{\Om}^j_{k^i}\|_{L^2}$ and $\| \hat{f}^j_{k^i}\|_{L\log L}$ are all uniformly bounded on their domains of definition, and moreover $\| \hat{f}^j_{k^i}\|_{L^1(U)} \to 0$ for any compact domain $U\in \R^2$. Thus by the compactness properties of the equation (see \cite{ST}, Theorem 1.2) and the singularity removal property (see \cite{LR}) we know that there is some map $w^j_i\in L^{\infty}$ with $\D w^j_i \in L^2$ and $\Om^j_i \in L^2$ solving 
$$-\Dl w^j_i = \Om^j_i \cdot \D w^j_i$$ and also $\hat{u}^j_{k^i} \to w^j_i$ strongly in $W^{1,2}_{loc}(\R^2\sm \{y^1,\dots y^l\})$ and $\hat{u}^j_{k^i} \to w^j_i$ is uniformly bounded in $W^{2,1}_{loc}(\R^2\sm \{y^1,\dots y^l\})$ - by Theorem 1.6 in \cite{ST}. We also know that (after a suitable choice of $\eps$) such solutions must have $\|\Om^j_i\|_{L^2}\geq 2\eps$ in order that $w^j_i$ is not a trivial (constant) solution - see \cite{LR}, Theorem 3.2.  

At this point we remark on the following improvements: 
\begin{itemize}
\item If $f_k\in L^p$ is uniformly bounded then we get $\hat{u}^j_{k^i} \to w^j_i$ strongly in $$W^{1,q}_{loc}(\R^2\sm \{y^1,\dots y^l\})$$ for all $q<\fr{2p}{2-p}$ (since the sequence $u_k$ will be uniformly bounded in $W^{2,p}_{loc}(\R^2\sm \{y^1,\dots y^l\})$). 
\item If we know $\Om_k\cdot \Db u_k = g_k \in L\log L$ is uniformly bounded then we get $\|\hat \Om_{k^i}^j \cdot \Db \hat u_{k^i}^j\|_{L^1} \to 0 $ locally in $\R^2$. Thus the limit bubble will satisfy $\Om^j_i \cdot \Db w^j_i = 0$ and is thus conformal. To see this notice that this condition implies the Hopf differential $\phi^j_i$ of $w^j_i$ is holomorphic and also $\|\phi^j_i\|_{L^1(\R^2)} < \infty$ giving that $\phi^j_i = 0$. 
\item In particular, using the above strong local convergence we end up with $\hat{\phi}^j_{k^i} \to 0$ locally strongly on $\R^2\sm  \{y^1,\dots y^l\}$ in $L^1$ if $f_k \in L\log L$, and in $L^{\fr{q}{2}}$ if $f_k \in L^p$ ($q>2$ as above). 
\end{itemize}

\paragraph{\bf{The covering argument}} 
We proceed by induction so first consider a single one of the $x^j$ as above. We know that there are finitely many ($Q^j$, say) point-scale sequences converging to this $x^j$. The aim of the argument below is to partition the set of bubble sequences in such a way as to separate different strings of bubbles forming at a point. In other words we separate out which bubbles are forming on which in order that the analysis and estimates do not interfere with the other strings of bubbles and we can reduce to an induction argument. 

%We say that a bubble forms on another bubble if their corresponding sequences satisfy:  
%$$\lim_{k\to \infty} \fr{|x^j_{k^s} - x^j_{k^i}|}{\rho^j_{k^s}} <\infty.$$
%Thus we must have $\fr{\rho^j_{k^s}}{\rho^j_{k^i}} \to \infty$ and 

First of all we re-label and order the $\rho^j_{k^i}$ so that $\rho^j_{k^1}\geq \rho^j_{k^2}\geq \dots \geq \rho^j_{k^{Q^j}}$. We partition the set according to the following scheme: %first off pick the largest scales $\rho^j_{k^1},\dots,\rho^j_{k^L}$ with $\fr{\rho^j_{k^i}}{\rho^j_{k^{i+1}}} \leq \Gamma <\infty$ for $i=1,\dots , L-1$. In other words we choose all the largest scales that form different bubbles at the same point (they do not form on each-other). Notice that by construction we have $\fr{|x^j_{k^i} - x^j_{k^l}|}{\rho^j_{k^1}} \to \infty$ for $i,l\leq L$ and for $L<s\leq Q$ we have $\fr{\rho^j_{k^i} }{\rho^j_{k^s}} \to \infty$. 

Single out the largest bubble scale $(x^j_{k^1},\rho^j_{k^1})$ (first re-label it $(x^j_{k^{i_1}},\rho^j_{k^{i_1}})$) and for the remaining point-scale sequences $\{(x^j_{k^i}, \rho^j_{k^i})\}_{i=2}^{Q^j}$ we consider 
$$\lim_{k\to \infty} \fr{|x^j_{k^i} - x^j_{k^{i_1}}|}{\rho^j_{k^{i_1}}}.$$ 
Let $i$ be the first scale such that this is infinite - we group this with the first scale and re-label it $i_2$. Notice that for any $i_1=1<s<i_2$ we have 
$$\lim_{k\to \infty} \fr{|x^j_{k^s} - x^j_{k^1}|}{\rho^j_{k^{i_1}}}<\infty\,\,\,\,\,\,\,\,\text{and}\,\,\,\,\,\,\,\,\,\,\,\,\,\,\,\, \lim_{k\to\infty}\fr{\rho^j_{k^{i_1}}}{\rho^j_{k^s}}=\infty$$ therefore there exists some $\Gamma<\infty$ such that 
$$B_{\rho^j_{k^s}}(x^j_{k^s})\In B_{\Gamma \rho^j_{k^1}} (x^j_{k^1})$$
for sufficiently large $k$. Thus we would say that the bubble $w^j_s$ associated with the scale $(x^j_{k^s},\rho^j_{k^s})$ forms on the bubble associated with $(x^j_{k^{i_1}},\rho^j_{k^{i_1}})$, $w^j_{i_1}$. Now, continue this procedure: for  $\{(x^j_{k^i}, \rho^j_{k^i})\}_{i=i_2+1}^{Q^j}$ we consider 
$$\lim_{k\to \infty} \fr{|x^j_{k^i} - x^j_{k^{i_1}}|}{\rho^j_{k^{i_1}}} \,\,\,\,\,\,\,\,\text{and}\,\,\,\,\,\,\,\,\lim_{k\to \infty} \fr{|x^j_{k^i} - x^j_{k^{i_2}}|}{\rho^j_{k^{i_2}}}.$$ 
Let $i_3$ be the first scale such that both of these are infinite and again notice that for $i_2<s<i_3$ the bubble $w^j_s$ forms on either $w^j_{i_1}$ or $w^j_{i_2}$. 

Continue this procedure until we exhaust all the bubble sequences. We are left with the following strings of bubbles 
$$\{(w^j_{i_1}, \{w^j_{i_1^s}\}_{s=1}^{J_1}), (w^j_{i_2}, \{w^j_{i_2^s}\}_{s=1}^{J_2}),\dots (w^j_{i_L}, \{w^j_{i_L^s}\}_{s=1}^{J_L})\}$$
where for each $i_l$ the bubbles $\{w^j_{i_l^s}\}_{s=1}^{J_l}$ all form on $w^j_{i_l}$ and there is some uniform $\Gamma<\infty$ such that $$\cup_{s=1}^{J_l} B_{\rho^j_{k^{i_l^s}}}(x^j_{k^{i_l^s}})\In B_{\Gamma \rho^j_{k^{i_l}}} (x^j_{k^{i_l}})$$ for sufficiently large $k$. Call $\{w^j_{i_l}\}_{l=1}^L$ the initial bubbles - the ones closest to the body map upon which all other bubbles form in the bubble tree.

%We do this in ascending order again: so let $t$ be the first point-scale sequence for which this is is infinite for all $i\leq L$; we add this to the first partition an re-label it with an $i=L+1$ (switch the $(L+1)^{th}$ sequence with the $t^{th}$ sequence). Keep doing this but this time we check if 
%$$\lim_{k\to \infty} \fr{|x^j_{k^s} - x^j_{k^i}|}{\rho^j_{k^s}}=\infty$$  for $1\leq i \leq L+1$ and $s\geq t+1$.  Continue until we have exhausted the point-scale sequences - say we now have $L+I = T$ point-scale sequences in the first partition.

%We end up with the following partition: 
%$$\{(x^j_{k^i}, \rho^j_{k^i})_{i=1}^{T}, (x^j_{k^s}, \rho^j_{k^s})_{s=T+1}^Q \}$$ which have 
%$$\lim_{k\to \infty} \fr{|x^j_{k^l} - x^j_{k^i}|}{\rho^j_{k^l}+\rho^j_{k^i}} = \infty $$ for all $1\leq i\neq l \leq T$. 
%Moreover for any $T+1 \leq s \leq Q$ we have both 
%$$\fr{\rho^j_{k^i}}{\rho^j_{k^s}} \to \infty \,\,\,\,\,\,\text{and}\,\,\,\,\,\, 
%\fr{|x^j_{k^i}- x^j_{k^s}|}{\rho^j_{k^s}} \leq \Gamma <\infty$$
% for some $i\leq T$. In other words there exists some $\Gamma<\infty$, and $i\leq T$ such that $B_{\rho^j_{k^s}}(x^j_{k^s})\In B_{\Gamma \rho^j_{k^i}} (x^j_{k^i})$ for all $s\geq T+1$.   
 
%For the scales
%$$\{(x^j_{k^s}, \rho^j_{k^s})\}_{s=L+1}^I$$ we do the same as above: order them so they are in descending order and partition the scales such that 

Now we find a covering argument for the collection of initial bubble domains:
$$\{B_{\Gamma \rho^j_{k^{i_l}}}(x^j_{k^{i_l}}) \}_{l=1}^L$$ 
thus ensuring that all bubble domains are covered. Let 
$$R^j_{k^1}:=2\sup_{l,t} |x^j_{k^{i_t}} - x^j_{k^{i_l}}|\to 0$$
%and 
%$$R^j_{k^{i_l}}:= 2 |x^j - x^j_{k^{i_l}}| \to 0$$ 
and notice that by construction $\fr{R^j_{k^1}}{\rho^j_{k^i_l}} \to \infty$ for all $i_l$ and in particular there exists some $x^j_{k^1}\to x^j$ such that $$\cup_{l=1}^L B_{\Gamma \rho^j_{k^{i_l}}}(x^j_{k^{i_l}}) \In B_{R^j_{k^1}} (x_{k^1}^j).$$

Now we know that there exists some $\dl>0$ such that 
$$\sup_{R^j_{k^1}<\rho<\fr{\dl}{2}} \int_{B_{2\rho}(x_{k^1}^j)\sm B_{\rho}(x_{k^1}^j)} |\Om_k|^2 < \eps.$$
If this is not the case then we rescale by the radius $\rho_0$ for which
\[
 \int_{B_{2\rho_0}(x_{k^1}^j)\sm B_{\rho_0}(x_{k^1}^j)} |\Om_k|^2 \ge \eps
\]
and standard arguments then show that we have found a new point-scale sequence which is a contradiction. We call $B_{\dl}(x_{k^1}^j)\sm B_{R^j_{k^1}}(x_{k^1}^j)$ a neck domain - a degenerating annulus on which $\|\Om\|_{L^2}$ is small on every dyadic sub-annulus.

For two different sequences $x^j_{k^{i_l}}, x^j_{k^{i_t}}$ we consider $$ R^j_{k^{lt}}= \lim_{k\to \infty} \fr{R^j_{k^1}}{|x^j_{k^{i_l}}-x^j_{k^{i_t}}|} $$ and group together those for which $R^j_{k^{lt}} = \infty$ and for the rest we have $R^j_{k^{lt}} \leq M.$

Notice that there is at least one $R^j_{k^{lt}}$ which falls into the second category. Re-order so that the former have $1\leq l<t \leq I$ and the latter $I+1\leq t\leq L$. Now, there is a finite cover of $B_{R^j_{k^1}} (x_{k^1}^j)$ by balls of radius $\fr{R^j_{k^1}}{2M}$ for which each ball in the cover satisfies exactly one of the following conditions: 
\begin{itemize}
\item It contains a single initial bubble domain, i.e. it equals $B_{\fr{R^j_{k^1}}{2M}}(x^j_{k^{i_l}})$ when $I+1\leq l\leq L$ and it's Hausdorff distance to any other bubble domain is at least $\fr{R^j_{k^1}}{4M}$. 
\item It contains finitely many (but more than one - $L_b$, say)  initial bubble sequence for $1\leq l\leq I$ - label these balls $B_{\fr{R^j_{k^1}}{2M}}(y_{k^b})$
\item It's Hausdorff distance from all the bubble domains is at least $\fr{R^j_{k^1}}{4M}$, labelled $B_{\fr{R^j_{k^1}}{2M}}(e_{k^b})$ - we call these empty domains as they must have $$\|\Om_k\|_{L^2(B_{\fr{3R^j_{k^1}}{4M}}(e_{k^b}))} <\eps$$ (otherwise it would be a new point-scale sequence).  
\end{itemize}
For the first option, we are again in the situation where (upon possibly increasing the value of $\Gamma$) 
$$\sup_{\Gamma \rho^j_{k^{i_l}}<\rho<\fr{R^j_{k^1}}{4M}} \int_{B_{2\rho}(x^j_{k^{i_l}})\sm B_{\rho}(x^j_{k^{i_l}})} |\Om_k|^2 < \eps.$$ In other words  $$B_{\fr{R^j_{k^1}}{2M}}(x^j_{k^{i_l}})\sm B_{\Gamma \rho^j_{k^{i_l}}}(x^j_{k^{i_l}})$$ is a neck domain. 

For the middle option suppose there are $L_b>1$ bubble sequences in $B_{\fr{R^j_{k^1}}{2M}}(y_{k^b})$. Let 
$$R^j_{k^2} = 2\sup_{t<l\leq L_b} |x_{k^{i_t}}^j - x^j_{k^{i_l}}|\to 0$$ and run the argument as above but this time in the ball $B_{\fr{R^j_{k^1}}{2M}}(y_{k^b})$. This time we manage to partition the set of bubbles once again and we can cover the ball $\fr{R^j_{k^1}}{2M}(y_{k^b})$ by neck domains, empty domains and bubble domains.% single out at least one bubble sequence for $1\leq l\leq I$ and also cover the balls   

If we continue this argument we are eventually in the situation where the second option above cannot happen. Thus we have managed to cover the ball $B_{R^j_{k^1}} (x^j)$ in a finite number of empty domains, neck domains and initial bubble domains. 

On each of the initial bubble domains we can inductively start all over again from the beginning of the covering argument. 

It should be clear now that a simple induction argument allows us to split the whole ball $B_1$ into a sequence of finitely many bubble domains, neck domains, empty domains and body domains.

The main body domain is the region $B_1\sm \{\cup_j B_{\dl}(x_{k^1}^j)\}$ for some uniform $\dl$ - on which we have uniform control on our maps according to the $\eps$-regularity theory. The remaining body domains are the bodies of the  bubbles which appear as we continue down our induction argument. On the main body domain we have uniform convergence to the limit map $u$ (in the appropriate sense), and on the bodies of the bubbles we have uniform convergence to the bubble (again, in the appropriate sense).  %{\bf{Ben: Define neck, body, bubble regions, bubbles, scaling properties (maybe we should forget $L\log L$ here?)}}. 

The empty domains are of the form $B_{R^k}(x^k)$ for some $\{x^k\}\In B_1$ and $R^k\to 0$, moreover $\|\Om_k\|_{L^2(B_{R^k}(x^k))} <\eps$ and also the point-scale sequence $(x^k,R^k)$ is distinct from all others. We must have on each empty domain that $\|\D u_k\|_{L^2(B_{R^k}(x^k))} \to 0$ - since otherwise we could re-scale to find a new bubble and thus a new point-scale sequence that must carry away at least $\eps$ of $\|\Om\|_{L^2}$ (this follows from the gap result in Theorem 3.2 of \cite{LR}) - a contradiction. Moreover it is a consequence of the estimates in \cite{ST} that if $f_k$ is bounded in $L^p$ for some $p>1$ then $\|\D u_k\|_{W^{1,1}(B_{R^k}(x^k))} \to 0$.

The neck regions are therefore the only place we can lose track of our convergence and these are precisely the regions on which our main theorems apply. Moreover, each dyadic sub-annulus on a neck can be covered by finitely many empty regions, thus we have $\|\D u_k\|_{L^2} \to 0$ on such domains and if $f_k$ is bounded in $L^p$ for some $p>1$ then $\|\D u_k\|_{W^{1,1}} \to 0$.

\end{document}